\documentclass[12pt,reqno]{amsart}

\usepackage{color}
\usepackage{amsmath, amsthm, amssymb}
\usepackage{amsfonts}
\usepackage[ansinew]{inputenc}
\usepackage[dvips]{epsfig}
\usepackage{graphicx}
\usepackage[english]{babel}
\usepackage{hyperref}

 \usepackage{mathrsfs}

\newcommand{\scrP }{\mathscr{P}}

\usepackage{cite}
\usepackage{graphicx}
\usepackage{amscd}
\usepackage{color}
\usepackage{bm}
\usepackage{enumerate}

\usepackage{verbatim}
\usepackage{hyperref}
\usepackage{amstext}
\usepackage{latexsym}
\theoremstyle{plain}
\newtheorem{theorem}{Theorem}[section]

\newtheorem{corollary}[theorem]{Corollary}
\newtheorem{proposition}[theorem]{Proposition}
\newtheorem{lemma}[theorem]{Lemma}

\numberwithin{theorem}{section}
\numberwithin{equation}{section}

\newcommand{\average}{{\mathchoice {\kern1ex\vcenter{\hrule height.4pt
width 6pt depth0pt} \kern-9.7pt} {\kern1ex\vcenter{\hrule
height.4pt width 4.3pt depth0pt} \kern-7pt} {} {} }}

\def\R{\mathbb{R}}

\renewcommand{\a }{\alpha }
\renewcommand{\b }{\beta }
\renewcommand{\d}{\delta }

\newcommand{\e }{\varepsilon }
\newcommand{\g }{\gamma}

\newcommand{\G }{\Gamma}
\renewcommand{\l }{\lambda }
\renewcommand{\L }{\Lambda }
\newcommand{\n }{\nabla }
\newcommand{\vp }{\varphi }

\newcommand{\s }{\sigma }

\renewcommand{\t }{\tau }
\newcommand{\z }{\zeta}
\renewcommand{\th }{\theta }

\newcommand{\ov}{\overline}

\newcommand{\be}{\begin{equation}}
\newcommand{\ee}{\end{equation}}

\newcommand{\de}{\partial}

\newcommand{\ti}{\widetilde}

\newcommand{\calO }{\mathcal{O}}

\newcommand{\N}{\mathbb{N}}

\newcommand{\Z}{\mathbb{Z}}

\newcommand{\cB}{{\mathcal B}}

\newcommand{\cF}{{\mathcal F}}

\newcommand{\cH}{{\mathcal H}}

\newcommand{\cK}{{\mathcal K}}

\newcommand{\cN}{{\mathcal N}}
\newcommand{\cO}{{\mathcal O}}
\newcommand{\cP}{{\mathcal P}}

\newcommand{\cR}{{\mathcal R}}

\newcommand{\eps}{\varepsilon}

\renewcommand{\phi}{\varphi}

\newcommand{\bp}{{\textbf{p}}}

\renewcommand{\epsilon}{\varepsilon}

\renewcommand{\u}{{u}}    % changed
\begin{document}

\title[  periodic constant nonlocal mean curvature hypersurfaces]
{Multiply-periodic  hypersurfaces with constant nonlocal mean curvature }

\author[Ignace A. Minlend  ]
{Ignace Aristide Minlend }
\address{Ignace Aristide  Minlend: African Institute for Mathematical Sciences of Senegal,
KM 2, Route de Joal, B.P. 14 18. Mbour, S\'en\'egal}
\email{ignace.a.minlend@aims-senegal.org }

\author[A. Niang]
{Alassane Niang}
\address{Alassane Niang:  Cheikh Anta Diop University
B.P. 5005 Dakar-Fann, S\'en\'egal.}
\email{alassane4.niang@ucad.edu.sn}

\author[El hadji Abdoulaye Thiam  ]
{El hadji Abdoulaye Thiam}
\address{ El Hadji Abdoulaye Thiam: African Institute for Mathematical Sciences of Senegal,
KM 2, Route de Joal, B.P. 14 18. Mbour, S\'en\'egal.}
\email{elhadji@aims-senegal.org }


%\keywords{Nonlocal mean curvature,...}
%\subjclass[2010]{Primary ; Secondary}

\begin{abstract}
We study  hypersurfaces with fractional mean curvature in $N$-dimensional Euclidean space.
These hypersurfaces are   critical points of the fractional perimeter under a volume constraint.
We use local inversion arguments    to prove  existence of smooth branches of multiply-periodic hypersurfaces bifurcating from suitable parallel hyperplanes. 
\end{abstract}

\maketitle

\section{Introduction and main result}
\label{sec:main-result}
Let $\alpha \in (0,1),$ $N\geq 2$ and  $E$ be an open set  in $\R^N$ with $C^2$-boundary.  The Nonlocal Mean Curvature (NMC) of the set $E$ (or the hypersurface $\de E$) is  defined as
\begin{equation}
  \label{eq:def-frac-curvature}
  H_E(x):= \textrm{PV} \int_{\R^N}\frac{ \tau_{E^c}(y)}{|x-y|^{N+\alpha}}\,dy :=\lim_{\eps \to 0} \int_{|y-x| \ge \eps}
\frac{ \tau_{E^c}(y)}{|x-y|^{N+\alpha}}\,dy \qquad\textrm{  for every $x\in \de E$.}
\end{equation}
 Here and in the following, we use the notation
$$
\tau_{E^c}(y):=1_{  E^c}(y) -1_{E}(y),
$$
where  $1_A$ denotes the characteristic function of $A$ and $E^c:=\R^N\setminus E$.
In the first integral PV denotes the principal value sense.
%For the asymptotics $\alpha$ tending to 0 or 1, $H_E$ should be
The nonlocal mean curvature $H_E$, renormalized  with a positive constant factor $C_{N,\alpha}$, converges locally uniformly to the classical mean curvature, as $\a\to 1$, (see e.g. \cite{Davila2014B,Ambrosio, Abatangelo}).

The NMC defined  in \eqref{eq:def-frac-curvature}  enjoys  a geometric expression that can be derived  via  integration by parts. It is given by 

\begin{equation}\label{geometric-0}
H_E(x)= -\frac{2}{\alpha} \int_{\partial E} \frac{ (x-y)\cdot\nu_E(y)}{|x-y|^{N+\alpha}}\,dy,
\end{equation}
where $\nu_E(y)$ denotes the outer unit normal to $\partial E$ at $y$. 

Notice that the integral in \eqref{geometric-0} is absolutely convergent in the Lebesgue sense if, for $\b>\a$, $\de E$ is of class $C^{1,\b}$ and $\int_{\de E}\frac{dy}{\left(1+\vert y\vert\right)^{N+\a-1}}dy<+\infty$. Moreover with this geometric  expression, the NMC of any $C ^{1,\b}$ orientable hypersurface can be defined. We use both expressions \eqref{eq:def-frac-curvature} and \eqref{geometric-0} in the computations.\\

\textit{In this paper, we say that  a non empty set  $E$  is a Constant Nonlocal    Mean Curvature  (CNMC) set or $ E$ is a set with  CNMC, if $H_E: \de E \to \R$  is pointwisely equivalent to  a  constant function. In this case $\de E$ is called a CNMC hypersurface or a hypersurface with CNMC.}\\

It is  around 2008 that Caffarelli and Souganidis in \cite{Caff-Soug2010}, and Caffarelli, Roquejoffre, and Savin in \cite{Caffarelli2010} introduced the notion of non local mean curvature. See e.g. \cite{Caffarelli2010}, it is a geometric quantity that appears in the first variation of the fractional perimeter, \cite{FFMMM}.
In the recent years, several works have been devoted to the study of
nonlocal minimal surfaces, while there is a lot to be understood,
for instance   the classification of stable nonlocal minimal cones. We refer the reader to \cite{Cabre2015A} for a brief description of
known results.

In the case of constant nonlocal and \textit{nontrivial} mean curvature there have been few results that appeared in the litterature.  In  \cite{Cabre2015A},  Cabr\'e, Fall,  Sol\`a-Morales and Weth
established the analogue of the Alexandrov rigidity theorem for bounded CNMC hypersurfaces in $\R^N$. Namely the sphere is the only  smooth  and bounded CNMC hypersurface.  This result was proved at the same time and independently by  Ciraolo, Figalli, Maggi, and Novaga~\cite{Ciraolo2015}.

The question of existence of unbounded CNMC hypersurfaces (as defined above) has been  first studied  by Cabr\'e, Fall,  Sol\`a-Morales and Weth  in \cite{Cabre2015A}.
Indeed, in \cite{Cabre2015A}, the authors constructed a continuous branch of CNMC one-periodic bands in $\R^2$ bifurcating  from a straight
band $\{(s,\z) \in \R\times \R \, : \, |\z|<\l\}$, for some $\l>0$. This result  was  generalized and improved very recently in \cite{CFW} by  Cabr\'e, Fall
 and Weth. Indeed,  they   established in \cite{CFW}  a nonlocal analogue of the classical result of Delaunay \cite{Delaunay} on periodic cylinders with
constant mean curvature, the so called onduloids in $\R^N$, which are one-periodic. They bifurcate smoothly from a cylinder $\{(s,\z) \in \R\times \R^{N-1} \, : \, |\z|<\l\}$. \\
We note that the paper \cite{Davila2015}  by  D\'avila, del Pino, Dipierro and Valdinoci, uses variational methods to prove
the existence of   $1$-periodic hypersurfaces of Delauney-type in $\R^N$
which minimize a certain renormalized fractional perimeter under a volume constraint.
%More precisely, \cite{Davila2015}~establishes the existence of a $1$-periodic minimizer for every given volume within the
%slab $\{(s,\z)\in\R\times\R^{N-1}\, : \, -1/2<s<1/2\}$.
It is an open problem to know if the boundary of these minimizers are CNMC surfaces as defined above. In \cite{Cabre2015B},  Cabr\'{e}, Fall  and  Weth
constructed CNMC sets  in $\R^N$ which are the countable union of a
perturbed sphere and all its translations through a periodic integer
lattice of dimension $M\leq N.$ These CNMC sets form a $C^2$ branch
emanating from the unit ball alone, where the parameter in the
branch is the distance to the closest lattice point.

The  present paper  deals  with new kind of unbounded multiply-periodic CNMC hypersurfaces with nontrivial NMC.  Indeed,  we  build constant nonlocal mean curvature sets, 
which are $2\pi\Z^{N-1}\times \frac{1}{\t}\Z$-periodic,  bifurcating from the translation-invariant union of  parallel slabs $E_\l^\t:=\{(s,\z)\in \R^{N-1}\times \R\,:\,|\z|<\l\}+\frac{1}{\t} e_N\Z$,  for some $\l\in (0,\frac{1}{2\t})$ and $\t$ positive small.  

We note  that the hypersurfaces  $ \de E_\l^\t$ have  zero local (or classical) mean curvature. However, this is not the case for the nonlocal mean curvature as can be easily verified from the geometric form of the NMC \eqref{geometric-0}. In particular our CNMC hypersurfaces do not have counterparts in the classical  theory of constant mean curvature surfaces.

We consider $u: \R^{N-1} \to (0,\infty)$  a  $2\pi\Z^{N-1}$-periodic  function and  $e_N=(0,\dots,0,1)\in \R^N$.  We look for $2\pi\Z^{N-1}\times \frac{1}{\t}\Z$-periodic sets  with constant nonlocal mean curvature
which have the form
\begin{equation}
  \label{eq:cylindrical-graph}
E_u^\t= E_u+\frac{1}{\t} e_N\Z= \bigcup_{q\in \Z}\left( E_u+\frac{q}{\t} e_N\right),
\end{equation}
where \begin{equation}
  \label{eq:cylindrical-graph}
E_u=\{(s,\z) \in \R^{N-1}\times \R \, : \, |\z|<u (s)\},
\end{equation}
$\t>0$,   $e_N=(0,\dots,0,1)$ and $u\in C^{1,\b}(\R^{N-1})$ is a $2\pi\Z^{N-1}$-periodic function. \\

Provided $\|u\|_{L^\infty(\R^{N-1})}<1/2\t$, it is clear that $E^\t_{u}$ is a  set which is $2\pi\Z^{N-1}\times  \frac{1}{\t}\Z$-periodic. Moreover,  for  every $\l\in (0,\frac{1}{2\t})$, the set $ E^\t_\l$ is  a CNMC set.
% Throughout the paper, $C^{k,\gamma}(\R^2)$ denotes the space of
%$C^k(\R^2)$ bounded functions $u$, with bounded derivatives up to order $k$ and with $D^k u $ having finite H\"older
%seminorm of order $\gamma\in (0,1)$.

To state our main results, we introduce first the function  spaces where we look for $u$. The space $C^{k,\gamma}_{p,e}(\R^{N-1})$ is the subspace of the H\"older space  $C^{k,\gamma}(\R^{N-1})$ constituted by functions which are $2\pi$-periodic and even in each of their variables. %We denote by $\cP$  the set of all permutation $(N-1)\times (N-1)$-matrix coordinate permutations of
We then define the spaces
\begin{align*}
&X_\cP:=\left\{ u\in C^{1,\b}_{p,e}(\R^{N-1})\,:\,  \quad\textrm{ $u$ is invariant under coordinate permutations} \right\},\\
&Y_\cP:=\left\{ u\in C^{0,\b-\a}_{p,e}(\R^{N-1})\,:\,  \quad\textrm{ $u$ is invariant under coordinate permutations} \right\}.
\end{align*}
 The above  spaces are equipped with their respective  standard H\"older norms, see \eqref{eq:def-hoelder-norm} below.
 
Our  main result is the following.
\begin{theorem}\label{th:lemellae}
Let  $N\geq 2$,  $\alpha\in(0,1)$  and   $\b\in (\a,1)$. There exist  $\l_*>0$,  $\t_0,b_0>0$  and a unique  $C^1$
curve
\begin{align*}
&[0,\t_0)\times (-b_0, b_0) \to X_\cP\times \R_+ , \qquad (\t, b) \mapsto ( w_{\t, b}, \lambda_{\t, b})
\end{align*}
with the following properties:
\begin{itemize}
\item[(i)] $  \l_{0,0}=\l_*$,  $w_{0, 0}\equiv\l_*$,
\item[(ii)] For all  $(\t,b) \in [0,\t_0)\times (-b_0, b_0)$, the domains  $$ E_{w_{\t,b}}^\t=\{(s,z)\in \R^{N-1}\times \R \, : \, |z|<w_{\t,b} (s)\}$$
are   constant non local mean curvature sets  and
$$H_{E_{w_{\t,b}}^\t}(x)=H_{E_{\l_{\t,b}}^\t} \qquad\textrm{ for all $x\in \de E_{w_{\t,b}}^\t$}.$$
\item[(iii)] For all  $(\t,b) \in [0,\t_0)\times (-b_0, b_0)$  and  $s=(s_1,\dots,s_{N-1})\in
\R^{N-1}$, we have
\begin{equation}\label{form}
  w_{\t, b}(s)=\l_{\t,b}+b \sum_{i=1}^{N-1}\cos\left(   s_i\right) +b v_{\t,b}( s).
\end{equation}
Furthermore, there exists a constant $c>0$  such that $$|\l_{\t,b}-\l_* |+ \| v_{\t,b}\|_{C^{1,\b}(\R^{N-1})}\leq
c(|\t|^{1+\a}+|b|)$$   and we have  $$\int_{[-\pi,\pi]^{N-1}} v_{\t,b}(s)\cos(s_i)\,
ds=0$$   for all  $i=1,\dots, {N-1}$   and  $(\t,b) \in
[0,\t_0)\times (-b_0, b_0)$.\\\
\end{itemize}
\end{theorem}
The nonlocal mean curvature operator  $H$ of  the set $E_u^\t$,  is given by
$$
H(\t,u)(s)= H_{E_u^\t}(s, u(s)) \qquad \textrm { for all $s\in \R^{N-1}$}.
$$
The  result in Theorem \ref{th:lemellae} is  obtained by solving the equation 
\begin{equation}\label{eqnmc}
H(\t,\l +\vp )-H(\t, \l)=0 \quad \textrm{ on} \quad Y_\cP,
\end{equation}
for some $\l\in \R$ and $\vp \in X_\cP$. This is achieved  by means of the implicit function theorem. To proceed, we  first provide the expression of the NMC,  $H_{E_u}$, of the set 
${E_u}$ and we compute its  linearization at any function $u$. Next  we study  the spectral properties  of the linearized operator $\mathcal{H}_{\l}$ with respect to constant  functions $u\equiv \l>0$.  The operator  $\mathcal{H}_{\l}$ has a non trivial $(N-1)$-dimensional  kernel.  However, its restriction  on the subspace  of functions that are even in each variable and  invariant under coordinates permutations yields an operator with  $1$-dimension kernel  and whose eigenvalues are increasing as functions in the variable $\l.$  Combining these properties with regularity  results contained in \cite{Fallll}  and   \cite{Sil},  we  gather  the hypotheses that enable us to apply the implicit function theorem.\\

Let us notice that the equation \eqref{eqnmc}   is valid as  well for $\t=0$ and  $H(0, u)$ is  nothing, but the  non local mean curvature of the set $E_u$, (see Section \ref{s:exist-lamellae}). As a direct  consequence of our construction, we get in Corollary  \ref{th:lemellae22} below, the existence of a smooth branch of constant nonlocal mean curvature set  which are $2\pi\Z^{N-1}$-periodic and bifurcate   from the slab  $E_{\l_*}:=\{(s,\z)\in \R^{N-1}\times \R\,:\,|\z|<\l_*\}$,  where $\l_*$ is the parameter in Theorem  \ref{th:lemellae}.  In the particular case when $N=2$,  we  also recover the branch of CNMC  hypersurfaces constructed   in \cite{Cabre2015A}.

\begin{corollary}\label{th:lemellae22}
Let  $N\geq 2$,  $\l_*$  and $b_0$  be the positive constants  in Theorem \ref{th:lemellae}. Then there  exists a unique  $C^1$
curve
\begin{align*}
&(-b_0, b_0) \to X_\cP\times \R_+ , \qquad  b\mapsto ( w_{ b}, \lambda_{b})
\end{align*}
with the following properties:
\begin{itemize}
\item[(i)] $  \l_{0}=\l_*$,  $w_{0}\equiv\l_*$,
\item[(ii)] For all  $ b \in (-b_0, b_0)$, the domains  $$ E_{w_{b}}=\{(s,z)\in \R^{N-1}\times \R \, : \, |z|<w_{b} (s)\}$$
are   constant non local mean curvature sets,  and
$$H_{E_{w_{b}}}(x)=H_{E_{\l_{b}}} \qquad\textrm{ for all $x\in \de E_{w_{b}}$}.$$
\end{itemize}
\end{corollary}
%It is an open problem to establish the existence of global continuous
%branches of the doubly-periodic  CNMC sets $E_{u_b}$ and the triply-periodic CNMC sets $E_{u_{\t,b}}$ given by Theorem~\ref{res:cyl1} and Theorem \ref{th:lemellae}
%and to study their limiting configuration. For instance,  in the case of classical mean curvature, embedded Delaunay hypersurfaces
%vary from a cylinder to an infinite compound of tangent spheres, \cite{Delaunay}. In the nonlocal case, the complete bifurcation from nonlocal Delauney hypersurfaces to a translation-invariant string of spheres  remains an open problem, see \cite{CFW}.  We believe that a  limiting configurations of our multiply-periodic slabs are translation-invariant rectangular  lattices of small perturbed spheres.

 The result in this paper parallels the result of Fall, Minlend and Weth in \cite{FMW}, where the authors built unbounded periodic domains to Serrin's overdetermined boundary value problem. The boundary of these domains bifurcate from generalized slabs.  We also quote the related  papers \cite{Sicbaldi2010} and \cite{Sicbaldi2012},  where it is
proved the existence of Delauney-type domains in $\R^N$ whose first Dirichlet
eigenfunction has constant Neumann data on the boundary.

Further  notable related results are contained in \cite{Fall-Polymer}, where the author considers equilibrium surfaces which are critical points to an energy constituted by the sum  of the area functional and a repulsive   Coulomb-type potential, arising in diblock copolymer melt. Among other results, the paper   \cite{Fall-Polymer} establishes the existence of multiply-periodic equilibrium patterns domains bifurcating from   parallel slabs. In  \cite{Davila2014B},    Davil\'a, del Pino and Wei proved the existence of  an unstable  two-sheet nonlocal minimal surface which is symmetric with respect to the horizontal plane, bifurcating from two parallel planes.  \\

The paper is organized as follows. In Section \ref{s:NMC-ope},  we provide the   expression of $H_{E_u}$ which will be used to prove the regularity of the NMC operator.    In Section \ref{sec:nonl-probl-solve}, we study the spectral properties of the  linearized operator at a  constant function $\l>0$ and provide qualitative properties of its eigenvalues, as $\l$ varies.  In Section \ref{s:exist-lamellae}, we prove Theorem \ref{th:lemellae} and Corollary  \ref{th:lemellae22}. Finally   the  proof of the regularity of the NMC operator, respectively   the computations of the explicit expression of the first derivative,  are done in Section \ref{sec:preparations-1}.   

\bigskip
\noindent \textbf{ Acknowledgements}: The  authors wish to thank Mouhamed Moustapha Fall  for his helpful suggestions and valuable comments  during the writing of this paper. The first and third  authors  are  funded  by the German Academic Exchange Service (DAAD,  
and the second  by the Non Linear Analysis, Geometry and Applications Project  (NLAGA).

\section{The NMC operator  of  perturbed slabs}\label{s:NMC-ope}
Let $\alpha \in (0,1)$ and $\beta \in (\alpha,1)$.
For a positive function $u \in C^{1,\beta}(\R^{N-1})$, we consider the set $E_u$ as defined in (\ref{eq:cylindrical-graph}).
We first recall  the following expression for the NMC of $E_u$:
\begin{equation}\label{geometric}
H_{E_{u}}(x)= -\frac{2}{\alpha} \int_{\partial E_{u}} |x-y|^{-(N+\alpha)} (x-y)\cdot\nu_{E_{u}}(y)\,dy;
\end{equation}
see e.g. \cite{Cabre2015A}. Here $\nu_{E_{u}}(y)$ denotes the unit outer normal of $ \partial E_{{u}}$ and $dy$ is
the  volume  element of $\de E_{{u}} $.  Next, we consider the open set
\begin{equation}
  \label{eq:def-cO}
\cO:= \left\{{u}  \in C^{1,\beta}(\R^{N-1}) \::\: \inf_{\R^{N-1}} {u}  >0\right\}.
\end{equation}
For ${u}  \in \cO$,  we consider the map $F_{u}:  \R^{N-1}\times \R  \to \R^N$ given by
$$
F_{u}(s, z)= (s,{u}(s)  z ).
$$
The boundary of $E_u$,
$$
\de E_{u}=\left\{(s,{u}(s) \s )\in \R^{N-1}\times \R \, : \, \s\in \{-1,1\}\right\},
$$
is parameterized  by the restriction of  $F_u $ to  $\R^{N-1}\times \{-1,1\}$.

\subsection{The NMC operator}
The following result provides the expression for the NMC of $E_{u}$ in terms of the above parametrization and the
function $u$.
%%%%%%%%%%%%%%%%%%%%%%%%%%%%%%%%
\begin{lemma}\label{lem:geomNMC}
Let ${u}\in \cO$. Then the nonlocal mean curvature $H_{E_{u}}$  ---that we will denote by $H({u})(s)$---
at a point $(s,{u}(s)\th )$, with $\th\in  \{-1,1\}$, does not depend on $\th$ and is given by
\begin{align}
  -\frac{\alpha}{2}  H({u})(s) =&    \int_{\R^{N-1}}
  \frac{
  u(s)-u(s-t)-t \cdot \n u(s-t)
}{
\{|t|^2+({u}(s)-{u}(s-t))^2  \}^{(N+\a)/2}
}
d t\label{eq:Geom-from-NMC} \\
& -   \int_{\R^{N-1}}
  \frac{
  u(s)+u(s-t)+  t \cdot \n u(s-t)
}{
\{|t|^2+({u}(s)+{u}(s-t))^2  \}^{(N+\a)/2}
}
d t  .  \nonumber
\end{align}
Moreover,  the two integrals above converge absolutely in  the Lebesgue sense.
\end{lemma}
\begin{proof}
	The proof is similar to the one in  \cite[Lemma 2.1]{CFW}. We then skip it.
\end{proof}
\section{Properties of the linearized NMC operator }
\label{sec:nonl-probl-solve}

\subsection{The linearized NMC operator}\label{s:non-to-solve}
In order to  compute the linearized nonlocal mean curvature operator, we need to introduce the functional spaces in which we work.  Consider the Banach spaces
$$%\begin{equation}\label{spX}
 C^{k,\g}_{p}(\R^{N-1})=\{  u\in C^{k,\g}(\R^{N-1}): \text{ $u$ is $2\pi\Z^{N-1}$-periodic }\}
$$%\end{equation}
and
$$%\begin{equation}\label{spX}
 C^{k,\g}_{p,e}(\R^{N-1})=\{  u\in C^{k,\g}(\R^{N-1}): \text{ $u$ is $2\pi\Z^{N-1}$-periodic and even in each variable }\},
$$
for $k\in \N$ and $\g\in (0,1)$.
The norms in two spaces above are the standard
 $C^{k,\g}(\R^{N-1})$-norms,  defined by
\begin{equation}
\label{eq:def-hoelder-norm}
\|u\|_{C^{k,\gamma}(\R^{N-1})} : = \sum_{j=0}^k \| D^j u \|_{L^\infty(\R^{N-1})} +  \sup_{\stackrel{s,r \in \R^{N-1}}{s \not = r}}\frac{|D^k u (s)- D^ku(r)|}{|s-r|^{\gamma}}.
\end{equation}
We denote by $\cP$ the set of $(N-1)\times (N-1)$-matrix of permutations.
We then define
\begin{align*}
&X_\cP:=\left\{ u\in C^{1,\b}_{p,e}(\R^{N-1})\,:\, u(s)=u(\bp(s)),\quad\textrm{ for all $s \in \R^{N-1}$ and $\bp\in \cP$} \right\},\\
&Y_\cP:=\left\{  \tilde u\in C^{0,\b-\a}_{p,e}(\R^{N-1})\,:\, \tilde u(s )= \tilde u(\bp(s)  ),\quad\textrm{ for all $ s\in \R^{N-1}$ and $\bp\in \cP$} \right\}.
\end{align*}
The norms of     $X_\cP$ and $Y_\cP$, are those  of  $C^{1,\b}(\R^{N-1})$ and   $C^{0,\b-\a}(\R^{N-1})$
respectively.

The map $H: \calO \subset C^{1,\beta}(\R^{N-1}) \to C^{0, \b-\a}(\R^{N-1})$ is smooth in $\cO$ by Proposition \ref{lem:diff-H},
and  clearly $H$ sends $X_\cP\to Y_\cP$ by change of variables.

We will use the expression \eqref{eq:Geom-from-NMC} to show  the smoothness of $H:\cO\to C^{0,\b-\a}(\R^{N-1})$. We now state the following fundamental result which will be proved in Section \ref{sec:preparations-1}.
\begin{proposition}\label{lem:diff-H}
The map $H: \calO \subset C^{1,\beta}(\R^{N-1}) \to C^{0, \b-\a}(\R^{N-1})$ is of class $C^{\infty}$.  Moreover for every $k\in \N$ and $u\in \cO$, there exists a constant  $c >0$, only depending on $\a,\b,N$ and $\inf_{\R^{N-1}} u$,  such that
\be\label{eq:est-deriv-H-k}
\|D^k H(u)\|\leq c(1+\|u\|_{C^{1,\beta}(\R^{N-1}) }) ^c.
\ee
In addition, if $u \in \calO$, $\l>0$ and  $v,w \in C^{1,\beta}(\R^{N-1})$, then
 \begin{align}\label{eq:dHuv}
  D H ( u)[v](s)&=2 PV\int_{\R^{N-1}}\frac{v(s)-v(s-t)}{(|t|^2+(u(s)-u(s-t))^2)^{\frac{N+\a}{2}}}\, dt \nonumber\\
 & \quad-2\int_{\R^{N-1}} \frac{v(s)+v(s-t)}{(|t|^2+(u(s)+u(s-t))^2)^{\frac{N+\a}{2}}}\, dt.
  \end{align}
\end{proposition}

Next we need to study properties of the family of linearized operators
$$%\begin{equation}
 % \label{eq:def-L-operator}
L_\l :=   \l^{1+\a}  DH(\l) \;\in \; \cB(C^{1,\beta}_{p,e}(\R^{N-1}),C^{0,\beta-\a}_{p,e}(\R^{N-1} )), \qquad \quad \l>0.
$$%\end{equation}
Here and in the following, $\cB(X,Y)$ denotes the space of bounded linear operators between Banach spaces $X$ and $Y$.
 \subsection{Study of the linearized NMC operator at constant functions}\label{ss:linearized-op}

We now study the behavior of the eigenvalues of the operator $L_\l$,  as  $\l>0$ varies.
By Proposition~\ref{lem:diff-H}, $L_\l$ is given by
\begin{equation}
   \label{eq:def-L-operator-1}
  L_\l v(s)= 2\l^{1+\a} \left(  PV\int_{\R^{N-1}} \frac{v(s)-v(s- t)}{| t|^{N+\a}}  \,d t-     \int_{\R^{N-1}} \frac{v(s)+v(s-t)}{(|t|^2+4\l^2)^{\frac{N+\alpha}{2}}}  \,dt \right).
\end{equation}
\begin{lemma}\label{lemmL0}
\label{sec:nonl-probl-solve-1}
Let   $\l>0$. The functions  $e_k \in C^{1,\beta}_{p,e}(\R^{N-1})$ given by
\begin{equation}
  \label{eq:def-e-k}
  s=(s_1,\dots, s_{N-1}) \mapsto e_k(s )=\prod_{i=1}^{N-1}\cos(k_is _i), \quad k=(k_1,\dots, k_{N-1}) \in \N^{N-1}
\end{equation}
are the eigenfunctions of $L_\l: C^{1,\beta}_{p,e}(\R^{N-1})\to C^{0,\beta-\a}_{p,e}(\R^{N-1} )$ with
\be \label{eq:L-lam-ek-eq-nuk-ek}
L_\l e_k=\nu(\l |k|) e_k,
\ee
where $|k|=\sqrt{k_1^2+\dots+k_{N-1}^2}$ and  $\nu:(0,+\infty)\to \R$ is    given by
\begin{align} \label{eq:lkhk}
\nu(R)=2 R^{1+\a} \left(\int_{\R^{N-1}}\frac{1-    \cos \left(   t_1   \right)   }{|{  t}|
 ^{{N+\alpha}}}d{ t}-  \int_{\R^{N-1}}\frac{1+     \cos \left(  t_1   \right)   }{(|{ t}|^2+4R^{2})^{\frac{N+\alpha}{2}}}d{ t} \right).
\end{align}
%where $a_\a:=2^{- \alpha}\G^2\left(\frac{-1-\a}{2}\right)\cos\left(\a\pi/{2} \right)$ and $b_\a=\frac{2^{1-\a}\pi}{1+\alpha} $.
Moreover,
\be \label{eq:nu-prime-positive}
 \nu'(R)>0 \qquad \textrm{  for every $R>0$,}
\ee
 \begin{align}
\label{eigen}
\lim_{R \to +\infty}\frac{\nu(R)}{R^{1+\a}}= 2\int_{\R^{N-1}}\frac{1-    \cos \left(   t_1   \right)   }{|{  t}|
 ^{{N+\alpha}}}d{ t} >0
\end{align}
and
 \begin{align}
\nu_0:=\lim_{R\to 0}\nu(R)=-4 \int_{\R^{N-1}}\frac{1   }{(|{ t}|^2+4 )^{\frac{N+\alpha}{2}}}d{ t}    <0  .
\label{assym}
\end{align}

\end{lemma}
\begin{proof}

 To get \eqref{eq:lkhk},  we  prove by induction on  $\ell\in \N$,  $\ell \geq 1$  that for any even function $E$ defined on  $\R^{\ell}$,
\begin{equation}\label{eq: pro induc}
\int_{\R^{\ell}}
E(t)\prod_{i=1}^{\ell}\cos(k_is_i-k_it_i)d{t}=\prod_{i=1}^{\ell}\cos(k_is_i)\int_{\R^{\ell}}
\prod_{i=1}^{\ell}\cos(k_it_i)E(t)d{t}.
\end{equation}
Recall the  fact that
\begin{equation}\label{eq: rapeel}
\cos(k_i s_i-k_i
t_i)=\cos(k_is_i)\cos(k_it_i)+\sin(k_is_i)\sin(k_it_i).
\end{equation}
For $\ell=1$, \eqref{eq: pro induc} is obviously true using
\eqref{eq: rapeel}, the evenness of  $E$ and the oddness of the sine
function.

For $\ell=2$, we have
\begin{align}\label{eq: pro in}
\int_{\R^{2}}
&E(t)\prod_{i=1}^{2}\cos(k_is_i-k_it_i)d{t}\nonumber\\
&=\int_{\R} \cos(k_2 s_2-k_2
t_2)\int_{\R}\biggl([\cos(k_1s_1)\cos(k_1t_1)+\sin(k_1s_1)\sin(k_1t_1)]E(t)
d{t_1}\biggl)d{t_2}\nonumber\\
&=\int_{\R} \cos(k_2 s_2-k_2
t_2)\biggl(\int_{\R}\cos(k_1s_1)\cos(k_1t_1)E(t)
d{t_1}\biggl)d{t_2},
\end{align}
where we have used the evenness of $E$ to get the last line.

Arguing as in  \eqref{eq: pro in}, we obtain
\begin{align}\label{eq: eee}
\int_{\R^{2}}
E(t)\prod_{i=1}^{2}\cos(k_is_i-k_it_i)d{t}&=\int_{\R}\cos(k_1s_1)\cos(k_1t_1)\biggl(
\int_{\R} \cos(k_2 s_2-k_2 t_2)E(t)
d{t_2}\biggl)d{t_1}\nonumber\\
&=\int_{\R}\cos(k_1s_1)\cos(k_1t_1)\biggl( \int_{\R} \cos(k_2
s_2)\cos(k_2 t_2)E(t)
d{t_2}\biggl)d{t_1}\nonumber\\
&=\cos(k_1s_1)\cos(k_2s_2)\int_{\R^2}
\cos(k_1 t_1)\cos(k_2 t_2)E(t) d{t}\nonumber\\
&=
\prod_{i=1}^{2}\cos(k_is_i)\int_{\R^{2}}\prod_{i=1}^{2}\cos(k_it_i)E(t)d{t}.
\end{align}
Let us now assume that \eqref{eq: pro induc} holds true for $\ell-1$, ($\ell \in \N$), then
\begin{align}\label{eq: pro indl}
&\int_{\R^{\ell}}E(t)\prod_{i=1}^{\ell}\cos(k_is_i-k_it_i)d{t}\nonumber\\
&=\int_{\R}\cos(k_\ell
s_\ell-k_\ell t_\ell)
\biggl(\int_{\R^{\ell-1}}E(t)\prod_{i=1}^{\ell-1}\cos(k_is_i-k_it_i)d t_{1}...d t_{\ell-1}\biggl)d t_\ell\nonumber\\
&=
\prod_{i=1}^{\ell-1}\cos(k_is_i)\int_{\R^{\ell-1}}\prod_{i=1}^{\ell-1}\cos(k_it_i)\biggl(\int_{\R}E(t)\cos(k_\ell s_\ell-k_\ell t_\ell)d t_\ell \biggl) d t_{1}...d t_{\ell-1}\nonumber\\
&=
\prod_{i=1}^{\ell}\cos(k_is_i)\int_{\R^{\ell}}E(t)\prod_{i=1}^{\ell}\cos(k_it_i)d
t.
\end{align}
With this we conclude that \eqref{eq: pro induc} holds true for all
$\ell\in \N$

Applying \eqref{eq: pro induc}  to the even functions  $t\mapsto |t|^{-N-\a}$ and $t\mapsto (|{t}|^2+4\lambda
^{2})^{-\frac{N+\alpha}{2}}$, we get
$$
\int_{\R^{N-1}}
\dfrac{e_{k}(s)-e_{k}(s-{t})}{|{t}|^{N+\alpha}}d{t}=e_{k}(s)\int_{\R^{N-1}}
\dfrac{1-e_{k}({t})}{|{t}|^{N+\alpha}}d{t},
$$
$$\int_{\R^{N-1}}\frac{e_k(s-{t})}{(|{t}|^2+4\lambda
^{2})^{\frac{N+\alpha}{2}}}d{t}=e_k(s)\int_{\R^{N-1}}\frac{e_k({t})}{(|{t}|^2+4\lambda
^{2})^{\frac{N+\alpha}{2}}}d{t}.
$$
and it follows that
\begin{equation}\label{eq:eival11}
L_\l e_k(s)= \l^{1+\a}\sigma_k(\l)e_k(s),
\end{equation}
 where  we have set
 \be \label{eq:defSigma-k-lambd}
 \sigma_k(\l):=2 \int_{\R^{N-1}} \dfrac{1-e_{k}({t})}{|{t}|^{N+\alpha}}d{t}-2
 \int_{\R^{N-1}}\frac{1+e_k({t})}{(|{t}|^2+4\lambda
^{2})^{\frac{N+\alpha}{2}}}d{t}.
 \ee
Now by the  change of variables $t=\frac{\ov t}{|k|}$, we get
$$
 \sigma_k(\l)=2|k|^{1+\a} \int_{\R^{N-1}} \dfrac{1- \prod_{i=1}^{N-1}\cos \left( \ov t_i \frac{k_i}{|k|} \right)    }{|{ \ov t}|^{N+\alpha}}d{ \ov t}-2 |k|^{1+\a}
 \int_{\R^{N-1}}\frac{1+   \prod_{i=1}^{N-1}\cos \left( \ov t_i \frac{k_i}{|k|} \right)   }{(|{\ov t}|^2+4|k|^2\lambda
^{2})^{\frac{N+\alpha}{2}}}d{\ov t}.
$$
We  claim that the first integral in the above expression does not depend on $k$, while the second one is  a function of $|k|$. We consider $f: S^{N-2}\to \R$ given by, for all $\th\in S^{N-2}$
$$
 f  (\th):= \int_0^\infty { \rho^{-1-\alpha}} \int_{S^{N-2}}  \left(1-  \prod_{i=1}^{N-1}\cos(\rho  (\s\cdot e_i)( \th\cdot e_i) )\right) \, d{ \s} d\rho.
 $$
Let  $\cR\in O(N-1)$ be a rotation  matrix   and denote by $\cR^t$ its transpose.  By a change of variables $\s=\cR \ti \s $, we get
\begin{align*}
 f  (\cR \th)&= \int_0^\infty { \rho^{-1-\alpha}} \int_{\cR^t S^{N-2}}  \left(1-  \prod_{i=1}^{N-1}\cos(\rho  (\ti \s\cdot  \cR^t e_i)( \th\cdot  \cR^t e_i) )\right) \, d{\ti \s} d\rho\\
 &= \int_0^\infty { \rho^{-1-\alpha}} \int_{S^{N-2}}  \left(1-  \prod_{i=1}^{N-1}\cos(\rho  (\s\cdot   e_i)( \th\cdot   e_i) )\right) \, d{ \s} d\rho\\
 &= f(\th).
\end{align*}
Since $ \cR$ is arbitrary, it then follows that
$$
f(\th)=f(e_1).
$$
Using polar coordinates, we get
\begin{align*}
\int_{\R^{N-1}} \dfrac{1- \prod_{i=1}^{N-1}\cos \left( \ov t_i \frac{k_i}{|k|} \right)    }{|{ \ov t}|^{N+\alpha}}d{ \ov t}=   f  \left(\frac{k}{|k|} \right)= f(e_1) .
\end{align*}
That is
\begin{align*}
\int_{\R^{N-1}} \dfrac{1- \prod_{i=1}^{N-1}\cos \left( \ov t_i \frac{k_i}{|k|} \right)    }{|{ \ov t}|^{N+\alpha}}d{ \ov t}= \int_{\R^{N-1}} \dfrac{1-  \cos \left( \ov t_1  \right)    }{|{ \ov t}|^{N+\alpha}}d{ \ov t}.
\end{align*}
By a similar reasoning, we also have
$$
\int_{\R^{N-1}}\frac{1+   \prod_{i=1}^{N-1} \cos \left( \ov t_i \frac{k_i}{|k|} \right)   }{(|{\ov t}|^2+4|k|^2\lambda
^{2})^{\frac{N+\alpha}{2}}}d{\ov t}= \int_{\R^{N-1}}\frac{1+     \cos \left( \ov t_1   \right)   }{(|{\ov t}|^2+4|k|^2\lambda
^{2})^{\frac{N+\alpha}{2}}}d{\ov t}.
$$
The last two equalities above allows to  see that  $\l^{1+\a} \sigma_k(\l)=\nu(\l|k|) $,  where  $\nu$ is the function defined in  \eqref{eq:lkhk}. This also yields  \eqref{eq:L-lam-ek-eq-nuk-ek}.

  Now \eqref{eq:nu-prime-positive} and \eqref{eigen} follows immediately from the expression of $\nu$. Finally, to see \eqref{assym}, we observe that
$$
R^{1+\a}\int_{\R^{N-1}}\frac{1+     \cos \left( \ov t_1   \right)   }{(|{\ov t}|^2+4 R
^{2})^{\frac{N+\alpha}{2}}}d{\ov t}= \int_{\R^{N-1}}\frac{1+     \cos \left( R t_1   \right)   }{(|{  t}|^2+4
 )^{\frac{N+\alpha}{2}}}d{  t},
$$
where we made the change of variable $\ov t = R t$.

\end{proof}
In the following, we consider the fractional Sobolev spaces
\begin{equation}
  \label{eq:def-hpe}
H^{\s}_{p,e} := \Bigl \{v \in H^{\s}_{loc}(\R^{N-1}) \::\: \text{ $v$ is $2\pi\Z^{N-1}$-periodic and even in each variable}\Bigl \}
\end{equation}
for $\s\geq 0$, and we put $L^2_{p,e}:= H^{0}_{p,e}$. Note that $L^2_{p,e}$ is a Hilbert space with scalar product
$$
(u,v) \mapsto \langle u,v \rangle_{L^2_{p,e}} := \int_{[-\pi,\pi]^{N-1}} u(t)v(t)\,dt \qquad \text{for $u,v \in L^2_{p,e}$.}
$$
We denote the induced norm by $\|\cdot\|_{L^2_{p,e}}$. Since $\|e_k\|_{L^2_{p,e}}=\pi^{\frac{{N-1}}{2}}$, it follows that the set $\{\frac{e_k}{\pi^{\frac{{N-1}}{2}}},\:\,k\in \N^{N-1}\}$ forms  a complete orthonormal basis of $L^2_{p,e}$. Moreover,
$H^\s_{p,e} \subset L^2_{p,e}$ is characterized as the subspace of all functions $v \in L^2_{p,e}$ such that
$$
\sum_{k \in \N  ^{N-1}} (1+|k|^2)^{\s} \langle v, e_k \rangle_{L^2_{p,e}} ^2 < \infty.
$$
Therefore, $H^\s_{p,e}$ is also a Hilbert space with scalar product
\begin{equation}
  \label{eq:scp-hj}
(u,v) \mapsto  \sum_{k \in \N^{N-1}} (1+|k|^2)^{\s} \langle u, e_k \rangle_{L^2_{p,e}}  \langle v, e_k \rangle_{L^2_{p,e}}  \qquad \text{for $u,v \in H^\s_{p,e}$.}
\end{equation}

We  consider   the eigenspaces $V_\ell$ corresponding to the eigenvalues $\nu(\l\ell)$ of the operator $L_\l$,  defined as
\begin{equation}
  \label{eq:def-Vell}
V_\ell:=\textrm{span}\left\{  {e_k}  \;:\; |k|=\ell \right\} \,\subset \,\bigcap_{j \in \N} H^j_{p,e}.
\end{equation}
We also denote their orthogonal  $L^2$-complements by
$$
V_\ell^\perp:=\left\{ w\in L^{2}_{p,e}\,: \, \int_{[-\pi,\pi]^{N-1}}v(s)w(s)\,ds=0\quad\textrm{ for every $v\in V_\ell$}\right\}.
$$
Obviously, by Lemma \ref{lemmL0},
\be\label{eq:-Lv-Fpurier-slab}
 L_\l v =\nu(\l\ell) v \qquad \textrm{ for every $v\in V_\ell$}.
\ee
%
%We recall the  spaces  introduced in the first section,
%\begin{align}\label{eq:def-X-Y---cP-slab}
%&X_{\cP}:= \{\vp\in  X \::\: \vp(t_1,t_2)= \vp(t_2,t_1)  \text{ for all $(t_1,t_2) \in \R^2$} \},\nonumber\\
%&Y_{\cP}:= \{\vp \in  Y \::\:  \vp(t_1,t_2)= \vp(t_2,t_1)  \text{ for all $(t_1,t_2) \in \R^2$}\}.
%\end{align}
We note that
\be\label{eq:V_1-cap-XcP}
V_1\cap X_\cP=V_1\cap Y_\cP =\textrm{span}\{\ov{v}\}\qquad   \textrm{ with }\,\,\,\ov{v}(t)=\sum_{i=1}^{N-1} \cos(t_i).
\ee

The following proposition contains importants results that we will use  for the proof of Theorem  \ref{th:lemellae} in the next section. 
To apply the implicit  function theorem, we will need the following result.
\begin{proposition} \label{prop:All-need-to-CR}
 There exists a unique $\l_*>0$, only depending  on $\a$, $\b$ and $N$, such that the linear operator $L_*:= L_{\l_*}: X_\cP \to Y_\cP $  has the following properties.
\begin{itemize}
\item[(i)] The kernel $N(L_*)$ of $L_*$ is spanned by the function
  \begin{equation}
    \label{eq:def-v-0}
\ov{v} \in X_{\cP}, \qquad \ov{v}(t_1,\dots, t_{N-1})= \cos(t_1) +\dots + \cos(t_{N-1}).
  \end{equation}
  \item[(ii)]     $L_*: X_\cP\cap V_1^\perp  \to Y_\cP\cap V_1^\perp$ is an isomorphism.

\end{itemize}
Moreover
\begin{equation}
  \label{eq:transversality-cond}
\partial_\l \Bigl|_{\l= \l_*} L_\l \ov{v}=\nu'(\l_*)\ov{v} \not  \in Y_\cP\cap V_1^\perp.
\end{equation}
 \end{proposition}
\begin{proof}
By Lemma~\ref{lemmL0}, there exists a unique $\l_*>0$ such that $\nu(\l_*)=0$ and $\nu'>0$ on $(0,+\infty)$. This with \eqref{eq:-Lv-Fpurier-slab} and \eqref{eq:V_1-cap-XcP} imply that $N(L_*)=\textrm{span}\{\ov{v}\}$.  This proves (i) and \eqref{eq:transversality-cond}.

To prove (ii) we  pick $g\in Y_\cP\cap V_1^\perp\subset L^2_{p,e} $. Then
by \eqref{eq:-Lv-Fpurier-slab}, the asymptotics \eqref{eigen}
and the definition of fractional Sobolev spaces in terms of Fourier coefficients,
we deduce that there exists a unique $w\in H^{1+\a}_{p,e}$ with
\be \label{eq:w-perp-ov-v}
\int_{[-\pi,\pi]^{N-1}} \ov{v}(s) w(s)\,ds = 0
\ee
  such that
\begin{equation}\label{invert}
L_*w =  g .
 \end{equation}
By the uniqueness, we also have $w(s )=w(\bp(s) )$, for every $s\in \R^{N-1}$ and $\bp\in \cP$.\\
We define $P:\R^{N-1}\to \R$ by
\begin{equation}\label{eq:deP}
P(s):= \frac{  1}{(|s|^2+4\l^2_*)^{(N+\a)/2}}
\end{equation}
and
$$
c_*:=\int_{\R^{N-1}} P(s)\, d s.
$$
Then \eqref{invert}   can be written as
$$
    PV\int_{\R^{N-1}} \frac{w(s)-w(s- t)}{| t|^{N+\a}}  \,d t=     c_* w(s)+P\star w (s)  + \frac{1}{2 \l^{1+\a}_*}g(s)  \qquad\text{for $s \in \R^{N-1}$,}
$$
where $\star$ denotes convolution product. This is equivalent to,
\begin{equation}\label{eq:fracLaplac}
(-\Delta)^{\sigma}w(s)-c_* w(s)=f(s) \qquad\text{for  $s \in \R^{N-1}$},
\end{equation}
where $\sigma:=\frac{1+\alpha}{2}$ and 
\begin{equation}\label{eq: deff}
f:=P\star w+\frac{1}{2 \l^{2\sigma}_*}g.
\end{equation}
We now use the results in \cite{Fallll}  and   \cite{Sil} to prove that $ w\in
C^{1,\b} (\R^{N-1})$,   $\beta\in (\alpha, 1)$. \\

We start by proving that  $P\star w \in L^{\infty}(\R^{N-1})$.  
Remark that the function $P$ defined in \eqref{eq:deP}  belongs  to $L^2(\R^{N-1})$ to   $w\in H^{1+\alpha}\subset L^2_{loc}\R^{N-1})$.  

Let $y\in  [-\pi, \pi]^{N-1}$. We have
\begin{align*}\label{eq: coparconvol}
|P\star w(y)|&\leq \sum_{k\in \Z^{N-1}} \int_{ [-\pi, \pi]^{N-1}}P(y-z+2\pi k)|w(z)|\textrm{d} z\\
&\leq\sum_{k\in \Z^{N-1}} \biggl(\int_{ [-\pi, \pi]^{N-1}}P^2(y-z+2\pi k)\textrm{d} z\biggl)^{1/2} ||w||_{L^2([-\pi, \pi]^{N-1})}.
\end{align*}

Since $y, z\in  [-\pi, \pi]^{N-1}$, it follows that
\begin{align*}
|y-z+2\pi k|^2+4\lambda_*^2\geq C(N)|k|^2,\quad  \textrm{for}  \quad |k|\geq2N,
\end{align*}
where $C(N)$ is a positive constant only depending on $N$. 

It is also plain  that $$|y-z+2\pi k|^2+4\lambda_*^2\geq 4\lambda_*^2\quad  \textrm{for all }  \quad k\in \Z^{N-1}.$$
Thus 
\begin{equation}\label{eq:norminfin}
|P\star w(y)| \leq C(N, \lambda_*, w, \alpha)\biggl( \sum_{k\in \Z^{N-1}, |k| <2N} + \sum_{k\in \Z^{N-1}, |k|\geq 2N}\frac{1}{|k|^{N+\alpha}}\biggl),
\end{equation}
where  $C(N, \lambda_*, w, \alpha)$ is a  positive constant depending   on $N$, $\alpha$, $w$ and $\lambda_*$. The function  $P\star w $ being  also $2\pi \Z^{N-1}$-periodic, we deduce from \eqref{eq:norminfin} that $P\star w \in L^{\infty}(\R^{N-1})$ and by  \cite{Fallll},  it follows that  $w\in  C^{0, \beta-\alpha}(\R^{N-1}).$ 

Next, we write   \eqref{eq:fracLaplac11}   on the form
\begin{equation}\label{eq:fracLaplac11}
(-\Delta)^{\sigma}w(s)=c_* w(s)+f(s) \qquad\text{for  $s \in \R^{N-1}$} ,
\end{equation}
 where $f$ is defined  by \eqref{eq: deff}.  Observe  also  that  $ P\star w \in C^{0,\beta- \alpha}(\R^{N-1})$, since $P\in L^1(\R^{N-1})$. We finally apply \cite[Proposition~2.1.8]{Sil} to get   $ w\in
C^{1,\b} (\R^{N-1})$, and \eqref{eq:w-perp-ov-v} yields   $w\in X_\cP\cap
V_1^\perp $  which completes the proof.
\end{proof}

\section{Proof of Theorem \ref{th:lemellae}}\label{s:exist-lamellae}
%Our aim is to solve the equation
%\be
%H_{E^\t_{u}}
%\ee
For $\t\in \R \setminus \{0\}$,  $u\in \cO$, with $\|u\|_{L^\infty(\R^{N-1})}<\frac{1}{2|\t|}$,   we  consider  the sets
\be
 \label{eq:E-tau-u-proof}
E^\t_{u}:= E_u+ \frac{1}{\t}e_N\Z=\bigcup_{q\in \Z}\left(   E_u+\left(0,\dots,\frac{q}{\t}\right)\right).
\ee
We recall that
$$
E_u=\{(s,\z) \in \R^{N-1}\times \R \, : \, |\z|<u (s)\}=\{(s,z u(s)) \in \R^{N-1}\times \R \, : \, |z|<1\} .
$$
 It is easy to see that $H_{E_{u}^\t}(x+\frac{q}{\t} e_N)=H_{E_{u}^\t}(x) $ for very $x\in \de E_u$ and $q\in \Z$.  Moreover,   it is clear that  in \eqref{eq:E-tau-u-proof}, we have  a disjoint union.

 \begin{proposition}\label{pro:lamel-perf-cnmc}
Let $\t\in \R \setminus \{0\}$,  $u\in \cO$ be such that  $\|u\|_{L^\infty(\R^{N-1})}<\frac{1}{2|\t|}$. Then
 for every $\th\in\{1,-1\}$ and $s\in \R^{N-1}$ , we have
     \begin{align}
 H_{E_{u}^\t}( F_{u}(s,\th ))= H(u)(s)
-2 \sum_{q\in \Z_* } \int_{ E_{1}}  \frac {u(s-t) \,d td z}{\{ |t|^2 +  ( u(s) -  u(s-t) z- \frac{q}{\t} )^2 \}^{(N+\alpha)/2}},
 \end{align}
  where   $\Z_*=\Z \setminus\{0\}$.
In particular,   if $|\l | <\frac{1}{2|\t|}$,  then $  E_{\l}^\t$ is a set with constant nonlocal  mean curvature.
\end{proposition}
  \begin{proof}
%It is clear that $E^\t_{u}$ is a $C^{1,\b}(\R^2)$ set which is $2\pi\Z^2\times  \frac{1}{\t}\Z$-periodic, provided $\|u\|_{L^\infty(\R^2)}<1/2\t$. %Moreover,  for for every $\l\in (0,\frac{1}{2\t})$, the set $\de E^\t_\l$ is  a CNMC surface.
Let $s\in \R^{N-1}$ and $\th\in \{-1,1\}$. Then putting  $x= F_{u}(s,\th )$, we then have
 \begin{align*}
  H_{E_{u}^\t}( F_{u}(s,\th ))
& =-\frac{2}{\alpha} \int_{\partial E_{u}^\t} \frac{(x-y)\cdot\nu_{E_{u}^\t}(y)}{{|x-y|^{N+\alpha}} }\,dy\\
&= -\frac{2}{\alpha} \sum_{q\in \Z } \int_{\partial E_{u}}  \frac {(x-y-\frac{q}{\t}e_N)\cdot\nu_{E_{u}}(y)}{|x-y -\frac{q}{\t}e_N |^{N+\alpha}}dy.
%
%&=  -\frac{2}{\alpha} \sum_{q\in \Z } \int_{\partial E_{u}}  \frac {(x-y-\frac{q}{\t}e_N)\cdot\nu_{E_{u}}(y)}{|x-y -\frac{q}{\t}e_N |^{N+\alpha}}dy.
 \end{align*}
 This implies that
  \begin{align*}
  H_{E_{u}^\t}( F_{u}(s,\th ))=  H(u)(s)
 -\frac{2}{\alpha}\sum_{q\in \Z_* } \int_{\partial E_{u}}  \frac {(x-y-\frac{q}{\t}e_N)\cdot\nu_{E_{u}}(y)}{|x-y -\frac{q}{\t}e_N |^{N+\alpha}}dy.
 \end{align*}
 Using the fact that
 $$\nabla_y\cdot\left\{(x-y- \frac{q}{\t}e_N )|x-y-\frac{q}{\t}e_N |^{-(N+\alpha)}\right\}=\alpha |x-y- \frac{q}{\t}e_N |^{-(N+\alpha)}$$
and the divergence theorem, we get
  \begin{align*}
 H_{E_{u}^\t}( F_{u}(s,\th ))=  H(u)(s)
-2 \sum_{q\in \Z_* } \int_{ E_{u}}  \frac {1}{|x-y -\frac{q}{\t}e_N |^{N+\alpha}}dy,
 \end{align*}
 where we have used the fact that $y\mapsto  |x-y- \frac{q}{\t}e_N |^{-(N+\alpha)}$ is smooth in $E_u$ for every $x\in \de E_u$.\\
 By change of  variables $y=F_u(\ov s,z)=(\ov s, z u(\ov s))$, we deduce that
$$
  H_{E_{u}^\t}( F_{u}(s,\th ))=  H(u)(s)
-2 \sum_{q\in \Z_* } \int_{ E_{1}}  \frac { u(\ov{s}) \, d\bar{s} d z}{|F_{u}(s,\th)  -F_{u}(\bar{s}, z) -\frac{q}{\t}e_N |^{N+\alpha}},
$$
 recalling that $E_1=\{(s,z)\in \R^{N-1}\times \R\,:\, |z|<1\}=\R^{N-1}\times (-1,1)$.
Since
$$
|F_{u}(s,\th)-F_{u}(\bar{s},z)-\frac{q}{\t}e_N|^2 = |s-\bar s|^2 + | u(s)\th -  u(\bar s) z-\frac{q}{\t} |^2,
$$
it follows that
$$
 H_{E_{u}^\t}( F_{u}(s,\th ))= H(u)(s)
-2  \sum_{q\in \Z_* } \int_{ E_{1}}  \frac {u(s-t) \,dt d z}{\{ |t|^2 + (u(s)\th -  u(s-t) z- \frac{q}{\t})^2\}^{(N+\alpha)/2}},
$$
 where we made the change of variable $\ov s= s-t $.
 By the change of variables $z\to-z$ and using that $\Z_*=-\Z_*$, we get
$$
 H_{E_{u}^\t}( F_{u}(s,\th ))= H(u)(s)
-2 \sum_{q\in \Z_* } \int_{ E_{1}}  \frac {u(s-t) \,d td z}{\{ |t|^2 +  ( u(s)\th +  u(s-t) z+ \frac{q}{\t} )^2 \}^{(N+\alpha)/2}}.
$$
From the two identities above, it is easy to see that
$$
H_{E_{u}^\t}( F_{u}(s,1 )) =  H_{E_{u}^\t}( F_{u}(s, -1 )) \qquad\textrm{ for every $s\in \R^{N-1}$.}
$$
This  implies, in particular, that  if $u \equiv \l$, a constant, then $  E_{\l}^\t$ is a CNMC set.
  \end{proof}

We will consider in the following of this section, the parameter $\l_*>0$ given by Proposition  \ref{prop:All-need-to-CR} and we put
$$
\t_1:=\frac{1 }{6 \l_*}.
$$
Then, for all $ \t\in (-\t_1,\t_1)$ and  $u\in B_{X_\cP}(\l_*, \l_*/2)\subset\cO\cap X_\cP$, we have
\be \label{eq:lowerboun-integrand-cH}
| u(s)  +  u(s-t) z- \frac{q}{\t}| \geq  \frac{|q|}{2\t_1} \qquad   \textrm{ for all $z\in [-1,1]$ and all $s,t\in \R^{N-1}$.  }
\ee
We  define  the map
$$
\cH:  (-\t_1,\t_1)\times B_{X_\cP}(\l_*, \l_*/2)\to  Y_\cP\in  C^{0,\b-\a}(\R^{N-1})
$$
given by, for every $ u\in B_{X_\cP}(\l_*,  {\l_*}/{2})$,
\begin{align*}
&\cH(\t, u)(s)
&= \begin{cases}\displaystyle\sum_{q\in \Z_* } \int_{ E_{1}}  \frac {u(s-t) \,d t d z}{\{ |t|^2 + (  u(s) -  u(s-t) z - {q}/\t )^2\}^{\frac{N+\alpha}{2}}}& \quad \textrm{if $\t\neq 0$, }  \vspace{3mm}\\
0 & \quad \textrm{if $\t= 0$. }
\end{cases}
\end{align*}

This implies that,  for every $\t\in  (-\t_1,\t_1)\setminus\{0\}$, $u\in B_{X_\cP}(\l_*,\l_*/2)$, $\th\in \{-1,1\}$ and $s\in \R^{N-1}$,
\be
 H_{E_{u}^\t}( F_{u}(s,\th ))= H(u)(s)- 2 \cH(\t, u)(s).
\ee
In the following result, we prove regularity estimates for the map $\cH$.
\begin{proposition}\label{lem:diff-cH-t-u}
  The maps $\cH$  is of class  $C^1$ on   $ (-\t_1,\t_1)\times B_{X_\cP}(\l_*, \l_*/2) $. Moreover  there exists a constant $c=c(N, \a,\b)>0$  such that for every $u\in B_{X_\cP}(\l_*, \l_*/2) $ and    $\t\in (-\t_1,\t_1)$,
\be
\|   \cH(\t, u)\|_{C^{0,\b-\a}(\R^{N-1})} \leq c   | \t|^{1+\a}.
\ee
In addition,  for every $v\in X_\cP$
\begin{align}
 D_u \cH(\t, u)[v](s)&=-\sum_{q\in \Z_* } \int_{\R^{N-1}}\frac{v(s)-v(s-t)}{(|t|^2+(u(s)-u(s-t)-q/\t)^2)^{\frac{N+\a}{2}}}\, dt \nonumber\\
 & \quad+ \sum_{q\in \Z_* } \int_{\R^{N-1}} \frac{v(s)+v(s-t)}{(|t|^2+(u(s)+u(s-t)-q/\t)^2)^{\frac{N+\a}{2}}}\, dt
 \label{eq:Du-cH-tau}
\end{align}
and
\be
\| D_u \cH(\t, u)[v]\|_{C^{0,\b-\a}(\R^{N-1})} \leq c  \|v\|_{C^{1,\b}(\R^{N-1})} | \t|^{1+\a}.
\ee

\end{proposition}
\begin{proof}
By \eqref{eq:lowerboun-integrand-cH},  it is   easy to see that $  \cH(\t, u)$ is of class $C^\infty$ in $ (-\t_1,\t_1)\setminus\{0\}\times B_{X_\cP}(\l_*, \l_*/2) $.   Also by \eqref{eq:lowerboun-integrand-cH},  there exists a constant $c>0$,   such that for every $u\in B_{X_\cP}(\l_*,\l_*/2) $,    $\t\in (-\t_1,\t_1)$ and $s\in \R^{ N-1}$,
\begin{align*}
| \cH(\t, u)(s)|& \leq  c  \sum_{q\in \Z_* }  \int_{ \R^{N-1}} \frac {   dt  }{\{ |t|^2 +  \frac{q^2}{\t^2}\}^{\frac{N+\alpha}{2}}}\|u\|_{C^{1,\b}(\R^{N-1})}\\
&\leq c   | \t|^{1+\a}\sum_{q\in \Z_* } | q|^{-1-\a} \|u\|_{C^{1,\b}(\R^{N-1})} \\
&\leq  c   | \t|^{1+\a}\|u\|_{C^{1,\b}(\R^{N-1})}.
\end{align*}
In addition, by the mean value property, we also have
\begin{align*}
|  \cH(\t, u)(s)- \cH(\t, u)(\ov s)|& \leq   c   | \t|^{1+\a} |s-\ov s| + c |s-\ov s|    \sum_{q\in \Z_* }  \int_{ \R^{N-1}} \frac {   \frac{|q|}{|\t|} dt  }{\{ |t|^2 +  \frac{q^2}{\t^2}\}^{\frac{N+2+\alpha}{2}}}\\
&\leq   c   |s-\ov s| | \t|^{1+\a} .
\end{align*}
We therefore get
\be \label{eq:est-cH-t-1plus-alph}
\|  \cH(\t, u) \|_{C^{0,\b-\a}(\R^{N-1})} \leq c    | \t|^{1+\a}.
\ee
Now,  letting $V_\e=u+\e v$, we have
\begin{align}
D_u  \cH(\t, u)[v](s)=\frac{d \cH(\t, u) (V_\e)}{d\e}\Big|_{\e=0}(s)&= \sum_{q\in \Z_* } \int_{\R^{N-1}}v(s-t) \int_{-1 }^1  G(z) \, dz   d t \nonumber\\
&-  \sum_{q\in \Z_* }\int_{\R^{N-1}}      \int_{-1 }^1\left\{ v(s)- v(s-t) z   \right\}  G'(z)    dz d t,  \label{eq:G-prim}
\end{align}
where
$$
G(z)=\frac{1}{(|t|^2+ ( u(s) -u(s-t) z- q/\t)^2  )^{\frac{N+\a}{2}}},
$$
which satisfies
 $$
G'(z)=( N+\a)\frac{u(s-t) ( u(s)-u(s-t) z -q/\t)}{(|t|^2+ (u(s)-u(s-t) z - q/\t)^2  )^{\frac{N+2+\a}{2}}}.
 $$
Using integration by parts in the last integral of \eqref{eq:G-prim},    we deduce that
\begin{align*}
D_u\cH(\t,u)[v](s)&= -  \sum_{q\in \Z_* }\int_{\R^{N-1}}    \Bigl(   \left\{v(s)- v(s-t)   \right\} G(1)   +  \left\{ v(s)+v(s-t)    \right\}  G(-1) \Bigr) dt\\
&=  -  \sum_{q\in \Z_* } \int_{\R^{N-1}} \frac{v(s)-v(s-t) }{(|t|^2+ ( u(s) -u(s-t) -q/\t )^2  )^{\frac{N+\a}{2}}}   d t\\
 &\quad+  \sum_{q\in \Z_* }\int_{\R^{N-1}} \frac{v(s)+v(s-t) }{(|t|^2+ ( u(s) +u(s-t)-q/\t )^2  )^{\frac{N+\a}{2}}}   d t.
\end{align*}
This gives \eqref{eq:Du-cH-tau}.
Moreover, it is easy to derive, from the above expression of $ D_u \cH(\t, u)$ and similar arguments as above, the estimate
$$
\| D_u \cH(\t, u)[v]\|_{C^{0,\b-\a}(\R^{N-1})} \leq c   \|v\|_{C^{1,\b}(\R^{N-1})} | \t|^{1+\a}.
$$
This and \eqref{eq:est-cH-t-1plus-alph},  imply that   $\cH$    is  of class $C^1$ in $ (-\t_1,\t_1)\times B_{X_\cP}(\l_*, \l_*/2)$.
\end{proof}
%
%%%%%%%%%%%%%%%%%%%%%%%%%%%%

 We are now in position to complete the proof of Theorem \ref{th:lemellae}. The argument we use is inspired  by the proof of the Crandall-Rabinowitz bifurcation theorem, \cite{CR-ARMA, Crandall1971}.
 \proof[Proof of Theorem  \ref{th:lemellae} completed]

  Our aim is to solve the equation
\begin{equation}\label{eq:with tau}
 H(\l + \vp)-2\cH(\t, \l+\vp)= H(\l)-2 \cH(\t,\l) \quad\textrm{  on $Y_\cP$,}
\end{equation}

For this, we define  the map 
$ \Psi: (-\t_1,\t_1)\times(3\l_*/4,5\l_*/4)\times B_{X_\cP}(0, \l_*/4)\to Y_\cP$,  given by
$$\Psi(\t, \l, \vp):= \l^{1+\a}\left\{  H(\l+\vp)-2 \cH(\t, \l+\vp )  \right\}$$  and  $$\overline{\Psi} (\t, \l, \vp):=\Psi(\t, \l, \vp)-\Psi(\t, \l, 0)$$  so that \eqref{eq:with tau}  becomes equivalent to 
\begin{equation}\label{eq:with tauaa}
 \overline{\Psi}(\t, \l, \vp)=0 \quad\textrm{  on $Y_\cP$}.
\end{equation}
It is clear that $\Psi$  and $\overline{\Psi}$  are
of class $C^1$ by Proposition~\ref{lem:diff-cH-t-u}.
We shall use the Implicit Function Theorem  to find a solution of the form $$\vp(s)= b (\ov{v}(s)+ v(s)),$$
where $\ov v$ is defined in \eqref{eq:V_1-cap-XcP}.
We will apply this theorem to   the new $C^1$-map
$$ f: (-\t_1,\t_1)\times  (-\beta_*,\beta_*)\times (3\l_*/4,5\l_*/4)\times B_{X_\cP}(0, \gamma_*)\cap  (X_\cP\cap V_1^\perp )\to Y_\cP,$$ given by
$$
f(\t, b,  \lambda, v):=
\left\{
  \begin{aligned}
 & \frac{ \overline{\Psi}(\t, \l, b (\ov{v}+v ))}{b}\quad  \textrm{if}\quad b\ne 0\\
 & D_\varphi \overline{\Psi}( \t, \lambda, 0)(\ov{v}+v)\quad  \textrm{if}\quad b=0,
  \end{aligned}
\right.
$$
where $$\gamma_*:=\|\ov v\|_{C^{1,\b}(\R^{N-1})} \quad \textrm{and}\quad \beta_*:=\frac{\l_*}{8\gamma_*}.$$   
Observe that $$\overline{\Psi}(0, \l, \vp)= \l^{1+\a}\left[H(\l+\vp)-H(\l)\right].$$ and $$f(0,0, \l, v)=D_\varphi \overline{\Psi}( 0, \lambda, 0)(\ov{v}+v)=\l^{1+\a} DH(\l)(\ov{v}+v)=L_{\l}(\ov{v}+v).$$
In fact, we have the following properties of the map $f$:
 \begin{itemize}
 \item[1.]  $f(0,0, \l_*, 0)=L_{\l_*}(\ov{v})=0.$
 \item[2.] 
The differential  of  $(\l, v)\mapsto   f(0,0, \l, v)$ at $(\l_*, 0)$ is the mapping
$$
G: \R \times (X_\cP\cap V_1^\perp) \rightarrow Y_\cP, \quad (\gamma, z) \mapsto L_{\l_*}(z)+ \gamma\partial_\l \Bigl|_{\l= \l_*} L_\l \ov{v}=L_{\l_*}(z)+\gamma \nu'(\l_*)\ov{v},
$$
where we have used \eqref{eq:transversality-cond}.
The mapping $G$ is an isomorphism by Proposition \ref{prop:All-need-to-CR} and its inverse is given by
$$ y \mapsto \biggl( \frac{(y,  \ov{v})}{\nu'(\l_*)||\ov{v}||^2}, \quad L^{-1}_{\l_*} \biggl[y-\frac{(y, \ov{v})}{||\ov{v}||^2} \ov{v}\biggl] \biggl).$$
 \end{itemize}
 
We   apply the implicit function theorem to  get  the existence of positives constants $\t_0,b_0>0$,
 only depending on $\a,\b,\l_*$, and unique $C^1$  curves
\begin{align*}
&(-\t_0,\t_0)\times (-b_0, b_0) \to (0,\infty) \times (X_\cP\cap V_1^\perp), \qquad (\t,b) \mapsto (\l_{\t,b}, v_{\t, b})
\end{align*}
such that for every $(\t,b)\in
(-\t_0,\t_0)\times (-b_0, b_0)$ and $b\ne 0,$  
\begin{equation}\label{eqimplicit}
\overline{\Psi}(\t, \l_{\t,b}, b (\ov{v}+v_{\t, b} ))=0.
\end{equation} 

Now,  for $b\ne 0$ and $s\in \R^{N-1}$
\begin{align*}
f(\t, b,  \l_*, 0)(s)&= \frac{ \overline{\Psi}(\t, \l, b \ov{v})}{b}(s)\\
&=  \int_0^1  D_\vp \overline{\Psi}(\t, \l_*,b \varrho \ov{v}) [\ov{v} ](s)\, d\varrho \\
&=  \l^{1+\a}_*\int_0^1 \left\{    DH(\l_* +\varrho  b\ov{v})[\ov
v](s)-2 D_u\cH(\t, \l_*+\varrho b\ov{v} )[\ov v](s)   \right\}\,
d\varrho.
\end{align*}
Therefore by Proposition  \ref{lem:diff-H}  and  Proposition
\ref{lem:diff-cH-t-u}, we get
$$
\| f(\t, b,  \l_*, 0)\|_{C^{0,\b-\a}(\R^{N-1})}\leq c
(|\t|^{1+\a}+|b|).
$$
It then follows that
$$
|\l_{\t,b}-\l_* |+ \| v_{\t,b}\|_{C^{1,\b}(\R^{N-1})}\leq  c
(|\t|^{1+\a}+|b|).
$$
Finaly, letting $w_{\t,b}:=\l_{\t,b}+b (\ov v+ v_{\t,b})$,
we thus get properties (i)-(iii) of Theorem \ref{th:lemellae}. \\

Finally, notice that  \eqref{eqimplicit} yields in particular,  $$ \l_{0,0}=\l_* \quad \textrm{and} \quad v_{0, 0}=0.$$   Moreover,  puting 
$\l_b:=\l_{0,b}, \quad v_{b}: =v_{0, b}  \quad \textrm{and} \quad u_{b}:=\l_{b}+b (\ov{v}+v_{b}),$  we get with \eqref{eq:with tau} that 
\begin{equation}\label{eq:with tauu}
 H[\l_b+ b (\ov{v}+v_{b} )]= H(\l_b)\quad\textrm{  on $Y_\cP$,}
\end{equation}
which shows that  for all $b \in (-b_0,b_0)$, the domains  $$ E_{u_b}=\{(s,z)\in \R^{N-1}\times \R \, : \, |z|<u_b (s)\}$$  are  constant nonlocal mean curvature sets, with
 $$H_{E_{u_b}}(x)=H_{E_{\l(b)}} \qquad\textrm{ for every $x\in \de E_{u_b}$}.$$ The proof of  Corollary  \ref{th:lemellae22} is also completed. 
\section{Regularity of the NMC operator} \label{sec:preparations-1}
The purpose of this section is to  prove the regularity of the NMC operator  and to derive the result of Proposition~\ref{lem:diff-H}.

The  proof of the diffrentiability of the NMC operator is inspired by \cite{CFW}. In \cite{CFW} and \cite{Cabre2015B}, the linearized NMC operators were computed  only at constant functions. Here, our arguments allow  for the computation of the linearized operator at any function. We believe that this might be of interest for the study of certain qualitatives  of our CNMC hypersurfaces,  e.g. the stability or global bifurcation branch.\\

We  recall that
\begin{align}
  -\frac{\alpha}{2}  H({u})(s) =&    \int_{\R^{N-1}}
  \frac{
  u(s)-u(s-t)-t \cdot \n u(s-t)
}{
\{|t|^2+({u}(s)-{u}(s-t))^2  \}^{(N+\a)/2}
}
d t\label{eq:Geom-from-NMC1} \\
& -   \int_{\R^{N-1}}
  \frac{
  u(s)+u(s-t)+ t \cdot \n u(s-t)
}{
\{|t|^2+({u}(s)+{u}(s-t))^2  \}^{(N+\a)/2}
}
d t  .  \nonumber
\end{align}
We will need to consider the family of maps $H_\e :\cO\to C^{0, \b-\a}(\R^{N-1})$, for $\e\geq0$,  given by
\begin{align}
  -\frac{\alpha}{2}  H_\e({u})(s) =&    \int_{\R^{N-1}}
  \frac{
  u(s)-u(s- t)- t \cdot \n u(s-t)
}{
\{| t|^2+({u}(s)-{u}(s-t))^2 +\e^2 \}^{(N+\a)/2}
}
d t \label{eq:Geom-from-NMC-eps} \\
& -   \int_{\R^{N-1}}
  \frac{
  u(s)+u(s- t)+  t \cdot \n u(s- t)
}{
\{| t|^2+({u}(s)+{u}(s- t))^2 +\e^2 \}^{(N+\a)/2}
}
d t  .  \nonumber
\end{align}
Here, we notice that  by  a  change of variables formula (see e.g. the proof of Lemma  \ref{lem:geomNMC}), we  have
\be \label{eq:NMC-geom-eps}
 -\frac{\alpha}{2} H_\e(u)(s) =  \int_{\partial E_u} \frac{(x-y)\cdot \nu_{E_u}(y)}{(|x-y|^2+\e^2)^{(N+\alpha)/2}} \,dy \quad \textrm{ for $x=(s,u(s)  )=F_u(s,1)$}.
\ee
We make the obvious observation that   to prove the regularity of $H$,    it suffices to consider $\d>0$ and to prove the regularity of
$$
H_\e: \cO_\d \to C^{0, \b-\a}(\R^{N-1}), \qquad \text{where $\cO_\d:= \left\{u \in C^{1,\beta}(\R^{N-1})\::\: \inf_{\R^{N-1}} u > \delta\right\}$,}
$$
for every $\e\geq0$.
We will prove uniform estimates (with respect to $\e$) for the higher derivatives of $H_\e$. This will allow us to get the general   expression for first  derivative of $H$  in Proposition \ref{lem:diff-H}.  This  will be completed in Section \ref{ss:derivatives-of-H} below.

We fix some notations that will be used throughout this section. For $i=1,2,3,4$,   we define the  maps $\L_i: C^{1,\b}(\R^{N-1})\times\R^{N-1}\times\R^{N-1} \to \R $, by
$$
\L_1(\phi,s,t) = \frac{\phi(s)-\phi(s-t)}{|t|} = \int_0^1 \n \phi (s-\rho  t)\cdot \frac{t}{|t|} d\rho ,
$$
\be \label{eq:def-Lamb-2}
\L_2(\phi, s,t)= \L_1(\phi,s,t)-  \n \phi (s-t)\cdot \frac{t}{|t|}  =  \int_0^1(\n \phi(s-\rho  t)-\n \phi(s- t) )\cdot \frac{t}{|t|} d\rho,
\ee
\be  \label{eq:def-Lamb-3}
\L_3(\phi, s,t)= \phi(s)+\phi(s-t)
\ee
and
\be \label{eq:def-Lamb-4}
\L_4(\phi, s,t)= \phi(s)+\phi(s-t) + t\cdot \n \phi(s-t).
\ee
Therefore \eqref{eq:Geom-from-NMC-eps}, becomes
\begin{align}\label{eq:good-expression2-NMC-eps-compact}
 -\frac{\alpha}{2} H_\e (u)(s)  =&    \int_{\R^{N-1}}
  \frac{
 \L_2(u , s,t)
}{| t|^{N-1+\a}
\{ 1+  \L_1(u, s,t)^2+ | t|^{-2}\e^2 \}^{(N+\a)/2}
}
d t   \\
& -   \int_{\R^{N-1}}
  \frac{
 \L_4(u, s,t)
}{
\{| t|^2+ \L_3(u, s,t)^2+ \e^2 \}^{(N+\a)/2}
}
d t  .  \nonumber
\end{align}
We observe that  for every $s,\ov s, t\in\R^{N-1}$, we have
\be\label{eq:Phis}
| \L_2(\phi, s,t)|\leq 2  \|\phi\|_{C^{1,\b}(\R^{N-1})}\, \min( |t|^\b    ,1)
\ee
and also
\be\label{eq:Phis1s2}
| \L_2(\phi, s,t)-\L_2(\phi, \ov s,t)|\leq  2 \|\phi\|_{C^{1,\b}(\R^{N-1})}\, \min( |t|^\b   , |s-\ov s|^\b).
\ee
Note also that for every $s,\ov s,t\in \R^{N-1}$ and $i=1,3,4$, we have
 \be\label{eq:est-denom-cK-s1-s2}
\left| \L_i(\phi,s,t)^2 -  \L_i(\phi,\ov s,t)^2   \right| \leq 2 \|\phi\|_{C^{1,\b}(\R^{N-1})}^2|s-\ov s|^\b
\ee
and
 \be\label{eq:est-denom-cK-s1-s2-Lam1-odd}
\left| \L_i(\phi,s,t)^2 -  \L_i(\phi,s,-t)^2   \right| \leq 2 \|\phi\|_{C^{1,\b}(\R^{N-1})}^2\min(|t|^\b,1).
\ee
For $ \varrho>-N$, we  let  $\cK_{\varrho,\e}, \ov{\cK}_{\varrho,\e}: C^{1,\b}(\R^{N-1})\times\R^{N-1}\times\R^{N-1} \to \R$ be defined by
\begin{equation}
\label{eq:defcKvarrho}
\cK_{\varrho,\e} ({u},s,t)=\frac{1}{
\left( 1+  \L_1(u, s,t)^2+ | t|^{-2}\e^2 \right)^{(N+\varrho)/2}}
\end{equation}
and
\begin{equation}
\label{eq:def-ov-cKvarrho}
\ov \cK_{\varrho,\e} ({u},s,t)=\frac{1}{\left( \displaystyle |t|^2+   \L_3(u,s,t)^2+ \e^2  \right)^{(N+\varrho)/2}},
\end{equation}
in such a way that
\begin{align}\label{eq:good-expression2-NMC-more-compact}
-\frac{ \a}{2} H_\e(u)(s)  =    \int_{\R^{N-1}} \frac{ \L_2(u , s,t)}{|t|^{N-1+\a}} \cK_{\a,\e} ({u},s,t)dt -   \int_{\R^{N-1}}    \L_4(u, s,t)\ov{ \cK}_{\a,\e} ({u},s,t) d t  .
\end{align}

Using this expression \eqref{eq:good-expression2-NMC-more-compact}, we shall show that $H_\e: \cO_\d\to C^{0,\b-\a}(\R^{N-1})$ is of class $C^\infty$ for every $\d>0$ and $\e\geq0$. \\

For a finite set $\cN$, we let $|\cN|$ denote the length (cardinal) of $\cN$. It will be understood that $ |\emptyset|=0$.
Let $Z$ be a  Banach space and $U$ a nonempty open subset of $Z$. If $T \in C^{k}(U,\R)$ and $u \in U$, then $D^kT(u)$ is a continuous
symmetric $k$-linear form on $Z$ whose norm is given by
$$
   \|D^{k}T ({u}) \|= \sup_{{u}_{1}, \dots,  {u}_{k}\in Z }
     \frac{|D^{k} T ({u})[u_1,\dots,u_k]| }{   \prod_{j=1}^k \|   {u}_{j} \|_{ Z }}    .
$$
%If  $T_1,\, T_2 \in C^k(U,\R)$,  then also $T_1 T_2 \in C^k(U,\R)$, and the $k$-th derivative of $T_1 T_2$ at $u$ is given by
%\be \label{eq:Dk-T1T2}
%D^k(T_1 T_2 )({u})[u_1,\dots,u_k]= \sum_{\cN \in  \scrS_k} D^{|\cN|} T_1({u})[u_n]_{n\in \cN} \,  D^{k-|\cN|} T_2({u}) [u_n]_{n\in \cN^c} ,
%\ee
%where  $\scrS_k $ is the set of subsets of $\{1,\dots, k\} $ and  $ \cN^c= \{1,\dots, k\}\setminus \cN $ for $\cN \in \scrS_k$.
 If  $L: Z\to \R$ is a linear map, we have
\be \label{eq:Dk-LT2}
D^{|\cN|}(L T )({u})[u_i]_{i\in\cN}= L({u})    D^{|\cN|} T({u})[u_i]_{i\in\cN} +
\sum_{j\in\cN} L({u}_j)   D^{|\cN|-1} T({u})  [u_i]_{\stackrel{i \in \cN}{ i\neq j}}.
\ee
We let $T$ be as above, $V \subset \R$ open with $T(U) \subset V$ and $g:V  \to \R$ be a  $k$-times differentiable map.  The Fa\'{a} de Bruno formula states that
\be
\label{eq:Faa-de-Bruno}
D^k( g\circ T)(u)[u_1,\dots,u_k]= \sum_{\Pi\in\scrP_k} g^{ (\left|\Pi\right|)}(T(u)) \prod_{P\in\Pi} D^{\left|P\right| }T(u)[u_j]_{j \in P} ,
 \ee
for $u, u_1,\dots,u_k  \in U$, where $\scrP_k$ denotes the set of all partitions of  $\left\{ 1,\dots, k \right\}$, see e.g. \cite{FaadeBruno-JW}.\\

For a function $u: \R^{N-1} \to \R$, we use the notation
$$
[u; s,r]:= u(s)-u(r )\qquad \text{for $s,r \in \R^{N-1}$,}
$$
and we note the obvious equality
\be\label{eq:uv-s_1s_2}
[uv; s, r] = [u;s, r]v(s) + u( r)[v;s, r] \qquad \text{for $u,v: \R^{N-1} \to \R$, $s, r \in \R^{N-1}$.}
\ee

We first give some estimates related to the kernel ${\cK}_{\varrho,\e}$ and $\ov{\cK}_{\varrho,\e}$ above.
\begin{lemma}\label{lem:est-cK-2D}
Let   $k \in \N  $,  $\d>0$, $\varrho>-N$ and $\b\in (0,1)$.
\begin{enumerate}
\item[(i)] There exists a constant $ c=c(\varrho,\b,k,N )>1 $
 such that  for all $\e \geq0$,  $s,r,t\in\R^{N-1}$ and ${u}\in C^{1,\b}(\R^{N-1})$, we have
   \be
   \label{eq:Dk-K-s}
\|  D_{u}^k {\cK}_{\varrho,\e} ({u},s,t    )   \|\leq
 {c(1+ \|{u}\|_{C^{1,\b}(\R^{N-1})} )^{c}   }   ,
 \ee
   \be
   \label{eq:Dk-K-s_1s_2}
\| [D_{u}^k {\cK}_{\varrho,\e} ({u},\cdot ,t );s,r] \|\leq
 {c(1+ \|{u}\|_{C^{1,\b}(\R^{N-1})} )^{c}   \, |s-r|^\b }
 \ee
 and
     \be  \label{eq:est-K-odd}
|  {\cK}_{\varrho,\e}({u},s,t    ) - {\cK}_{\varrho,\e}({u},s,-t    )   |\leq  c(1+ \|{u}\|_{C^{1,\b}(\R^{N-1})}^{c} )
 {   \min ( |t|^\b  , 1) }     .
 \ee
\item[(ii)] There exists  $c=c(\varrho,\b,k,\d ,N)>1$ such that  for all $\e\geq0$,  $s,r,t\in\R^{N-1}$ and ${u}\in \cO_\d$, we have
\be  \label{eq:Dk-oK-s}
\|  D_{u}^k \ov{\cK}_{\varrho,\e} ({u},s,t   )   \|\leq
\frac{c(1+ \|{u}\|_{C^{1,\b}(\R^{N-1})} )^{c}   }{    (1 + |t|^2)^{(N+\varrho)/2}    }  ,
\ee
\be   \label{eq:Dk-oK-s_1s_2}
\| [D_{u}^k \ov{\cK}_{\varrho,\e} ({u},\cdot ,t );s,r] \|\leq
\frac{c(1+ \|{u}\|_{C^{1,\b}(\R^{N-1})} )^{c}   \, |s-r|^\b }{    (1 + |t|^2)^{(N+\varrho)/2}}.
\ee
 \end{enumerate}
 \end{lemma}
\begin{proof}
Throughout this proof, the letter $c$ stands for different constants greater than one and depending only on $\varrho,\b,k$ and $\d$.  Since all the estimates  trivially holds true for $k=0$,  we will assume in the following that $k\geq 1$.
We define
$$
Q:  C^{1,\b}(\R^{N-1})\times \R^{N-1}\times \R^{N-1} \to \R,\qquad Q({u},s,t)= \L_1(u,s,t)^2+|t|^{-2}\e^2
$$
and
$$
g_\varrho \in C^\infty(\R_+, \R), \qquad
g_\varrho(x)= (1+x)^{-(N+\varrho)/2},
$$
so that
$$%\begin{align} \label{eq:def-ti-cK-ell}
{\cK}_{\varrho,\e}({u},s,t)=g_\varrho\left(  Q({u},s,t)  \right) .
$$%\end{align}
By \eqref{eq:est-denom-cK-s1-s2-Lam1-odd}, we then have
 \begin{align*}
\Bigl|{\cK}_{\varrho,\e}({u},s,t)-& {\cK}_{\varrho,\e}({u},s,-t)   \Bigr|=|  g_\varrho  ( Q({u},s, t   ) ) - g_\varrho  ( Q({u},s,- t   ) )   |\nonumber\\
&=\frac{N+\varrho}{2}\left|\int_0^1  \frac{   Q({u},s, t    )-  Q({u},s, -t   )   }{ (1+\tau Q({u},s,t   )+(1-\t)Q({u}, s,-t  )  )^{(N+\varrho+2)/2}}   \, d\tau    \right|  \\
&\leq c \|{u}\|_{C^{1,\b}(\R^{N-1})}^2\,        \min(|t|^\b , 1). \nonumber
 \end{align*}
This gives \eqref{eq:est-K-odd}.

 By \eqref{eq:Faa-de-Bruno} and recalling that $Q$  is quadratic in ${u}$,  we have
 \begin{align}
 D_{u}^k {\cK}_{\varrho,\e}({u},s,t)&[u_1,\dots,u_k] \nonumber\\
= &\sum_{\Pi\in\scrP_k^2}  g_\varrho^{(\left|\Pi\right| )  } ( Q({u},s,t)   ) \prod_{P\in\Pi} D^{\left|P\right| }_{u} Q({u},s,t)[u_j]_{j \in P},   \label{eq:Dk-K-s_1s_2-0}
 \end{align}
 where $\scrP_k^2$ denotes the set of partitions $\Pi$ of $\{1,\dots,k\}$ such that $1\leq |P| \le 2$ for every $P \in \Pi$. Hence by \eqref{eq:uv-s_1s_2} we have
 \begin{align}
 \label{eq:Dk-K-s_1s_2-1}
& \hspace{-8mm}\left[ D_{u}^k {\cK}_{\varrho,\e}({u}, \cdot,t   )[u_1,\dots,u_k] ;s,r\right]\\
 =& \sum_{\Pi\in\scrP_k^2}  \left[  g_\varrho^{(\left|\Pi\right| )  }  ( Q({u},  \cdot,t )   ); s,r \right]
 \prod_{P\in\Pi} D^{\left|P\right| }_{u} Q({u},s,t )[u_j]_{j \in P}\nonumber\\
 &+\sum_{\Pi\in\scrP_k^2}    g_\alpha^{(\left|\Pi\right| )  }  ( Q({u},r,t )   )
 \Bigl[     \prod_{P\in\Pi} D^{\left|P\right| }_{u} Q({u},  \cdot,t )[u_j]_{j \in P}  \:;\:s,r \Bigr].\nonumber
 \end{align}

 For $P \in \Pi$ with $1\leq |P| \le 2$,  by using inductively \eqref{eq:uv-s_1s_2} and  \eqref{eq:est-denom-cK-s1-s2},  we find that
 \begin{align}\label{eq:DP-Q-s_1s_2}
|[D^{|P|}_{u} &Q ({u},\cdot,t)[u_j]_{j \in P};s,r]|  \nonumber \\
&\leq  c(1+ \|{u}\|_{C^{1,\b}(\R^{N-1})}^{2} )  |s-r|^\b \prod_{j \in P} \|{u}_j\|_{C^{1, \b}(\R^{N-1})}
\end{align}
and
 \begin{align}\label{eq:DP-Q-s}
 |D^{|P|}_{u} Q ({u},s  ,t)[u_j]_{j\in P}|
\leq  c(1+ \|{u}\|_{C^{1,\b}(\R^{N-1})}^{2} )  \prod_{j \in P} \|{u}_j\|_{C^{1,\b}(\R^{N-1})}.
\end{align}
For $\ell \in \N  $ and $x>0$, we have
$$
  g^{(\ell)}_\varrho(x)=(-1)^{\ell}2^{-\ell} \prod_{i=0}^{\ell-1 } (N+\varrho +2 i) (1+x)^{-\frac{N+\varrho+2\ell}{2}}.
$$
Consequently, for every $ {u}\in \cO_\d$, using \eqref{eq:DP-Q-s_1s_2}, we  have the estimates
\begin{align}
&| \left[ g^{({\ell})}_{\varrho,\e}  ( Q({u},\cdot, t   ) ) ;s,r\right]| \nonumber \\
&\le \left |[ Q({u},\cdot  ,t   ) ;s,r]  \int_0^1  g^{({\ell}+1)}_\varrho  (\tau Q({u},s, t)+(1-\tau)Q({u},r, t))  d\tau \right|
   \nonumber\\
%&\leq  c(1 + \|{u}\|_{C^{1,\b}(\R^{N-1})}^2)\,    \, |s-r|^\b    \nonumber\\
&\leq  c(1 + \|{u}\|_{C^{1,\b}(\R^{N-1})})^c  { |s-r|^\b}  \label{eq:gQ-s1-s2}
\end{align}
and
\be
\label{eq:gQ-s}
|   g^{({\ell})  }_\varrho  ( Q({u} , \cdot,t,p    )) |\leq   1
\ee
for ${\ell}=1,\dots,k$.  Therefore by  \eqref{eq:Dk-K-s_1s_2-1}, \eqref{eq:DP-Q-s_1s_2}, \eqref{eq:DP-Q-s}, \eqref{eq:gQ-s1-s2} and \eqref{eq:gQ-s},
we obtain
$$
| \left[ D_{u}^k {\cK}_{\varrho,\e}({u}, \cdot,t    )[u_1,\dots,u_k] ;s,r\right]|\leq
 {c(1+ \|{u}\|_{C^{1,\b}(\R^{N-1})} )^{c}   \, |s-r|^\b }     \prod_{i=1}^k \|{u}_i\|_{C^{1, \b}(\R^{N-1})}.
$$
This yields \eqref{eq:Dk-K-s_1s_2}.
 Furthermore we easily deduce from (\ref{eq:Dk-K-s_1s_2-0}), \eqref{eq:DP-Q-s} and \eqref{eq:gQ-s}    that
$$
|  D_{u}^k {\cK}_{\varrho,\e}({u},s,t   )[u_1,\dots,u_k]   |\leq
 {c(1+ \|{u}\|_{C^{1,\b}(\R^{N-1})} )^{c}   }     \prod_{i=1}^k \|{u}_i\|_{C^{1, \b}(\R^{N-1})},
$$
yielding \eqref{eq:Dk-K-s}. The proof of (i) is completed.

We now prove (ii). We define
$$
\ov{Q}:  C^{1,\b}(\R^{N-1})\times \R^{N-1}\times \R^{N-1}  \to \R,\qquad \ov Q({u},s,t)=\vert t\vert^2 +  \L_3(u,s,t)^2   + \e^2
$$
and
$$
\ov g_\varrho \in C^\infty(\R_+, \R), \qquad
\ov g_\varrho(x)= x^{-(N+\varrho)/2},
$$
so that
$$%\begin{align} \label{eq:def-ti-cK-ell}
\ov{\cK}_{\varrho,\e}({u},s,t)=\ov g_\varrho\left(  \ov Q({u},s,t)  \right) .
$$
 By \eqref{eq:Faa-de-Bruno} and recalling that $\ov Q$  is quadratic in ${u}$,  we have
 \begin{align}
 D_{u}^k \ov{\cK}_{\varrho,\e}({u},s,t)&[u_1,\dots,u_k] \nonumber\\
= &\sum_{\Pi\in\scrP_k^2}\ov   g_\varrho^{(\left|\Pi\right| )  } ( \ov Q({u},s,t)   ) \prod_{P\in\Pi} D^{\left|P\right| }_{u}\ov  Q({u},s,t)[u_j]_{j \in P},   \label{eq:Dk-oK-s_1s_2-0}
 \end{align}
  Hence by \eqref{eq:uv-s_1s_2} we have
 \begin{align}
 \label{eq:Dk-oK-s_1s_2-1}
& \hspace{-8mm}\left[ D_{u}^k \ov{\cK}_\alpha({u}, \cdot,t    )[u_1,\dots,u_k] ;s,r\right]\\
 =& \sum_{\Pi\in\scrP_k^2}  \left[ \ov  g_\varrho^{(\left|\Pi\right| )  }  ( \ov Q({u},  \cdot,t)   ); s,r \right]
 \prod_{P\in\Pi} D^{\left|P\right| }_{u} \ov Q({u},s,t)[u_j]_{j \in P}\nonumber\\
 &+\sum_{\Pi\in\scrP_k^2}  \ov   g_\varrho^{(\left|\Pi\right| )  }  ( \ov Q({u},r,t)   )
 \Bigl[     \prod_{P\in\Pi} D^{\left|P\right| }_{u} \ov Q({u},  \cdot,t)[u_j]_{j \in P}  \:;\:s,r \Bigr].\nonumber
 \end{align}

 For $P \in \Pi$ with $1\leq |P| \le 2$,  by using inductively \eqref{eq:uv-s_1s_2} and  \eqref{eq:est-denom-cK-s1-s2},  we find that
 \begin{align}\label{eq:DP-oQ-s_1s_2}
|[D^{|P|}_{u} &\ov Q ({u},\cdot,t,p)[u_j]_{j \in P};s,r]|  \nonumber \\
&\leq  c(1+ \|{u}\|_{C^{1,\b}(\R^{N-1})}^{2} )(1+|t|^2) |s-r|^\b \prod_{j \in P} \|{u}_j\|_{C^{1, \b}(\R^{N-1})}
\end{align}
and
 \begin{align}\label{eq:DP-oQ-s}
 |D^{|P|}_{u} \ov Q ({u},s  ,t)[u_j]_{j\in P}|
\leq  c(1+ \|{u}\|_{C^{1,\b}(\R^{N-1})}^{2} )(1+|t|^2) \prod_{j \in P} \|{u}_j\|_{C^{1,\b}(\R^{N-1})}.
\end{align}
For $\ell \in \N $ and $x>0$, we have
$$
  \ov g^{(\ell)}_\varrho(x)=(-1)^{\ell}2^{-\ell} \prod_{i=0}^{\ell-1 } (N+\varrho +2 i) x^{-\frac{N+\varrho+2\ell}{2}}.
$$
Consequently, for every $ {u}\in \cO_\d$, using \eqref{eq:DP-oQ-s_1s_2}, we  have the estimates
\begin{align}
&| \left[ \ov g^{({\ell})}_\varrho  ( \ov Q({u},\cdot, t   ) ) ;s,r\right]| \nonumber \\
&\le \left |[ \ov Q({u},\cdot  ,t   ) ;s,r]  \int_0^1  \ov g^{({\ell}+1)}_\alpha  (\tau \ov Q({u},s, t)+(1-\tau)\ov Q({u},r, t))  d\tau \right|
   \nonumber\\
&\leq  c(1 + \|{u}\|_{C^{1,\b}(\R^{N-1})}^2)\,    \, |s-r|^\b (1+\vert t\vert^2) (|t|^2+ \d^2)^{-\frac{N+\varrho+2{\ell} + 2}{2}}  \nonumber\\
&\leq  c(1 + \|{u}\|_{C^{1,\b}(\R^{N-1})})^c  \frac{ |s-r|^\b}{ (1+|t|^2)^{\frac{N+\varrho+2{\ell}}{2}}} \label{eq:goQ-s1-s2}
\end{align}
and
\be
\label{eq:goQ-s}
| \ov   g^{({\ell})  }_\varrho  ( \ov Q({u} , \cdot,t    )) |\leq  \frac{c}{    (1 +|t|^2)^{(N+\varrho+2 {\ell} )/2}    }
\ee
for ${\ell}=0,\dots,k$.  Therefore by  \eqref{eq:Dk-oK-s_1s_2-1}, \eqref{eq:DP-oQ-s_1s_2}, \eqref{eq:DP-oQ-s}, \eqref{eq:goQ-s1-s2} and \eqref{eq:goQ-s},
we obtain
 \begin{align*}
&| \left[ D_{u}^k \ov{\cK}_{\varrho,\e}({u}, \cdot,t    )[u_1,\dots,u_k] ;s,r\right]|\\
 &\leq    c(1+ \|{u}\|_{C^{1,\b}(\R^{N-1})} )^{c}|s-r|^\b  \sum_{\Pi\in\scrP_k^2}       \frac{1}{    (1 + |t|^2)^{\frac{N+\varrho+2 \left|\Pi\right|}{2}}    }
 \prod_{P \in \Pi} (1+\vert t\vert^2) \prod_{j \in P} \|{u}_j\|_{C^{1,\b}(\R^{N-1})} \\
 &=      \frac{c(1+ \|{u}\|_{C^{1,\b}(\R^{N-1})} )^{c}   \, |s-r|^\b }{    (1+ |t|^2)^{(N+\varrho)/2}    }   \sum_{\Pi\in\scrP_k^2  }    \prod_{P\in\Pi}  \prod_{j \in P} \|{u}_j\|_{C^{1, \b}(\R^{N-1})}  .
 \end{align*}
We then conclude that
$$
| \left[ D_{u}^k \ov{\cK}_{\varrho,\e}({u}, \cdot,t    )[u_1,\dots,u_k] ;s,r\right]|\leq
\frac{c(1+ \|{u}\|_{C^{1,\b}(\R^{N-1})} )^{c}   \, |s-r|^\b }{    (1+ |t|^2)^{(N+\varrho)/2}    }     \prod_{i=1}^k \|{u}_i\|_{C^{1, \b}(\R^{N-1})}.
$$
This yields \eqref{eq:Dk-oK-s_1s_2}.
 Furthermore we easily deduce from (\ref{eq:Dk-oK-s_1s_2-0}), \eqref{eq:DP-oQ-s} and \eqref{eq:goQ-s}    that
$$
|  D_{u}^k \ov{\cK}_{\varrho,\e}({u},s,t   )[u_1,\dots,u_k]   |\leq
\frac{c(1+ \|{u}\|_{C^{1,\b}(\R^{N-1})} )^{c}   }{    (1+ |t|^2)^{(N+\varrho)/2}    }     \prod_{i=1}^k \|{u}_i\|_{C^{1, \b}(\R^{N-1})},
$$
completing the proof of (ii).
\end{proof}
%%%%%%%%%%%%%%%%%%%%%%%%
%
%
As in \cite{CFW},  we provide estimates for the   candidates to be the derivatives of $H$. Indeed,
letting
\be \label{eq:def-M-oM}
M_\e({u}, s,t)= \frac{1}{|t|^{N-1+\a}}\L_2(u,s,t) {\cK}_{\a,\e}(u,s,t) ,  \qquad\ov  M_\e({u}, s,t)=    \L_4(u,s,t) \ov{\cK}_{\a,\e}(u,s,t),
\ee
then  by  \eqref{eq:Dk-LT2}, if $k \ge 1$, we get
 \begin{align}\label{eq:DkM}
 D^{k}_{u} M_\e({u} , s,t)[u_i]_{i \in \{1,\ldots,k\}}
 &=  \frac{1}{|t|^{N-1+\a}} \L_2({u},s,t)    D^{k }_{u}\cK_{\a,\e}({u},s,t)[u_i]_{i \in \{1,\ldots,k\}} \nonumber \\
&\hspace{8mm} +  \sum_{j=1}^k \frac{1}{|t|^{N-1+\a}}  \L_2({u}_{j},s,t)   D^{k-1}_{u} \cK_{\a,\e}({u},s,t)
 [u_i]_{\stackrel{i \in \{1,\ldots,k\}}{ i\neq j}}
 \end{align}
 and
  \begin{align}\label{eq:DkoM}
 D^{k}_{u} \ov M_\e({u} , s,t)[u_i]_{i \in \{1,\ldots,k\}} &=   \L_4({u},s,t)
 D^{k }_{u}\ov \cK_{\a,\e}({u},s,t)[u_i]_{i \in \{1,\ldots,k\}} \nonumber \\
&\hspace{8mm} + \sum_{j=1}^k   \L_4({u}_{j},s,t)   D^{k-1}_{u} \ov \cK_{\a,\e}({u},s,t)
 [u_i]_{\stackrel{i \in \{1,\ldots,k\}}{ i\neq j}}.
 \end{align}
Here, we used that $u\mapsto \L_2(u,\cdot,\cdot)$ and $u\mapsto \L_4(u,\cdot,\cdot)$ are linear, see \eqref{eq:def-Lamb-2} and \eqref{eq:def-Lamb-4}.
Now our aim is to provide estimates for $s\mapsto   \int_{\R^{N-1}} D^{k}_{u} M_\e({u} , s,t)[u_i]_{i \in \{1,\ldots,k\}}(s) dt  $  and  $s\mapsto   \int_{\R^{N-1}}D^{k}_{u}\ov  M_\e({u} , s,t)[u_i]_{i \in \{1,\ldots,k\}}(s) dt$, from which the regularity of $H_\e$ will follow, and
 \begin{align*}
-\frac{\a}{2} D^k  H_\e({u})  (s) =&    \int_{\R^{N-1}}  D^kM_\e ({u},s,t)  d t+ \int_{\R^{N-1}} D^k \ov  M_\e ({u},s,t)  d t .
\end{align*}
These estimates will be uniform with respect to $\e$, since it will be need to  get the expression of   the derivatives of $H$   stated in  Proposition \ref{lem:linearized-NMC}. \\
%%%%%%%%%
We  will need the following technical result.
\begin{lemma}\label{lem:est-M-eps-ovM-eps}
Let  $N\geq 2$,   $k \in \N  $,  $\d>0$,   and $\a,\b\in (0,1)$. Then there exists a constant $ c=c(\a,\b,k,\d,N )>1 $ such that
 for all $\e \geq0$,  $s,r,t\in\R^{N-1}$ and ${u}\in C^{1,\b}(\R^{N-1})$, we have
\begin{align}\label{eq:estM-e1}
| D^{k}_{u} M_\e({u} , s,t)[u_i]_{i \in \{1,\ldots,k\}}  |\le  c(1+ \|{u}\|_{C^{1,\b}(\R^{N-1})} )^{c}   \frac{\min(|t|^{\b},1)    }{   |t|^{N-1+\a}  }
 \,    \prod_{i=1}^k \|{u}_i\|_{C^{1, \b}(\R^{N-1})} ,
\end{align}
\begin{align}\label{eq:estM-e2}
&|[D^{k}_{u} M_\e({u} , \cdot,t)[u_i]_{i \in \{1,\ldots,k\}}; s ,r]|\le  c(1+ \|{u}\|_{C^{1,\b}(\R^{N-1})} )^{c} \;\times
 \nonumber \\
&\hspace{10mm}\Bigl(\frac{\min( |t|^\b    , |s-r|^\b)}{  |t|^{{N-1}+\a}   }
+ \frac{\min(|t|^{\b},1)   |s-r|^\b }{   |t|^{{N-1}+\a}  }\Bigr)
 \,    \prod_{i=1}^k \|{u}_i\|_{C^{1, \b}(\R^{N-1})} ,
\end{align}
\begin{align}\label{eq:estoM-e1}
| D^{k}_{u} \ov M_\e({u} , s,t)[u_i]_{i \in \{1,\ldots,k\}}  |\le  c(1+ \|{u}\|_{C^{1,\b}(\R^{N-1})} )^{c}   \frac{\min(|t|^{\b},1)    }{ ( 1+ |t|^2  )^{(N+\a)/2}  }
 \,    \prod_{i=1}^k \|{u}_i\|_{C^{1, \b}(\R^{N-1})}
\end{align}
and
\begin{align}\label{eq:estoM-e2}
&|[D^{k}_{u}\ov  M_\e({u} , \cdot,t)[u_i]_{i \in \{1,\ldots,k\}}; s,r]|\le  c(1+ \|{u}\|_{C^{1,\b}(\R^{N-1})} )^{c} \;\times
 \nonumber \\
&\hspace{5cm}  \frac{\min(|t|^{\b},1)   |s-r|^\b }{   (1+  |t|^2  )^{(N+\a)/2}  }
 \,    \prod_{i=1}^k \|{u}_i\|_{C^{1, \b}(\R^{N-1})} .
\end{align}

\end{lemma}
\begin{proof}
By \eqref{eq:Phis} and Lemma \ref{lem:est-cK-2D}(i), the estimate \eqref{eq:estM-e1} follows. Using  inductively \eqref{eq:uv-s_1s_2},  the estimates \eqref{eq:Phis}, \eqref{eq:Phis1s2} and Lemma \ref{lem:est-cK-2D}(i),  we     get \eqref{eq:estM-e2}.   By similar arguments, the proof of \eqref{eq:estoM-e1} and  \eqref{eq:estoM-e2}  follow from   Lemma \ref{lem:est-cK-2D}(ii), \eqref{eq:est-denom-cK-s1-s2} and \eqref{eq:uv-s_1s_2}.
\end{proof}
%%%%%%%%%%%%%%%%%%%%%%%%%%%%
We prove estimates for all possible candidates for the derivatives of $H_\e$.
\begin{lemma}\label{lem:est-cand-deriv}
Let  $N\ge 2$,  $\d>0$, $\a\in (0,1)$, $\b\in (\a,1)$,  $\e\geq 0$, ${u} \in \cO_\d$ and $k\in \N $. For  ${u}_1,\dots, {u}_k \in C^{1,\b}(\R^{N-1})$, we  define the functions $\cF_\e, \ov{\cF}_\e: \R^{N-1}\to \R$ by
\begin{align*}
\cF_\e(s)&=\int_{\R^{N-1}}  D^kM_\e ({u},s,t)  d t
\end{align*}
and
\begin{align*}
\ov  \cF_\e(s)&= \int_{\R^{N-1}} D^k \ov  M_\e ({u},s,t)  d t.
\end{align*}
Then $\cF_\e\in C^{0,\beta-\alpha}(\R^{N-1})$ and $\ov  \cF_\e \in C^{0,\beta}(\R^{N-1})$. Moreover, there exists a constant $c=c(\a,\b,k,\d,N)>1$ such that  for every $\e\geq0$,
\begin{equation}
  \label{eq:est-F1}
\|\cF_\e\|_{C^{0,\beta-\alpha}(\R^{N-1})} \leq  c(1+ \|{u}\|_{C^{1,\b}(\R^{N-1})} )^{c}     \prod_{i=1}^k \|{u}_i\|_{C^{1, \b}(\R^{N-1})}
\end{equation}
and
\begin{equation}
  \label{eq:est-F2}
\|\ov  \cF_\e\|_{C^{0,\beta}(\R^{N-1})} \leq  c(1+ \|{u}\|_{C^{1,\b}(\R^{N-1})} )^{c}     \prod_{i=1}^k \|{u}_i\|_{C^{1, \b}(\R^{N-1})}.
\end{equation}
\end{lemma}
\begin{proof}
 Throughout this proof, the letter $c$ stands for different constants greater than one and depending only on $\alpha,\b,k $, $\d$ and $N$.  By \eqref{eq:estM-e1},  for every $s\in \R^{N-1}$,
\be
\label{eq:est-cF1-L-infty}
|\cF_\e(s) | \leq   c(1+ \|{u}\|_{C^{1,\b}(\R^{N-1})} )^{c} \|u\|_{C^{1,\b}(\R^{N-1})}\,     \prod_{i=1}^k \|{u}_i\|_{C^{1, \b}(\R^{N-1})}.
\ee
By \eqref{eq:estM-e2}, we have
\begin{align}
\label{eq:est-Fi-s1s2}
 &  | [\cF_\e; s,r] |   \le c(1+ \|{u}\|_{C^{1,\b}(\R^{N-1})} )^{c} \,
 \prod_{i=1}^k \|{u}_i\|_{C^{1, \b}(\R^{N-1})} \, \times \\
 &\qquad\qquad
\int_{\R^{N-1}}\left\{  \min( |t|^\b    , |s-r|^\b)+\min(|t|^{\b} ,1)  |s-r|^\b    \right\}      |t|^{-({N-1})-\a}dt.   \nonumber
\end{align}
Assuming $|s-r|\leq 1$, then using polar coordinates $t=\rho \th$, with $\th\in S^{N-2}$,   we get
\begin{align*}
 \left\{ \int_{|t|\leq |s-r|} + \int_{|t|\geq |s-r|}\right\}
 &\left\{  \min( |t|^\b    , |s-r|^\b)+\min(|t|^{\b} ,1)  |s-r|^\b    \right\}      |t|^{-N+1-\a}dt\\
 %
 %&= |S^{N-2}| \left\{  \min( r^\b    , |s-r2|^\b)+\min(r^{\b} ,1)  |s-r|^\b    \right\}      |\rho^{-1-\a}d\rho\\
 %
  \leq&|S^{N-2}|  \int_{0}^{  |s-r|}    \rho^{\b-\a-1}    d\rho
 +  |S^{N-2}|  |s-r|^\b  \int_{  |s-r|}^{+\infty}         \rho^{-1-\a}d\rho \\
 \leq & c|s-r|^{\b-\a}.
\end{align*}
Using this in \eqref{eq:est-Fi-s1s2},  we then conclude that
\be\label{eq:est-cF1}
 | [\cF_\e; s,r] | \le c(1+ \|{u}\|_{C^{1,\b}(\R^{N-1})} )^{c}  \,  |s-r|^{\beta-\alpha } \prod_{i=1}^k \|{u}_i\|_{C^{1, \b}(\R^{N-1})}.
\ee
Letting $\vert s-r\vert\geq 1$, we have, by \eqref{eq:est-cF1-L-infty}
\begin{eqnarray}\label{eq:est-cF2}
	| [\cF_\e; s,r] | &\leq & 2\parallel\cF_\e\parallel_{L^\infty(\R^{N-1})}\nonumber\\
	& \leq & c(1+ \|{u}\|_{C^{1,\b}(\R^{N-1})} )^{c}\vert s-r\vert^{\b-\a} \prod_{i=1}^k \|{u}_i\|_{C^{1, \b}(\R^{N-1})}.
\end{eqnarray}

\eqref{eq:est-cF1} and \eqref{eq:est-cF2} together with \eqref{eq:est-cF1-L-infty} give (\ref{eq:est-F1}).

To prove (\ref{eq:est-F2}), we just integrate the  inequality in \eqref{eq:estoM-e2} and \eqref{eq:estoM-e1} on $\R^{N-1}$.
 \end{proof}

We are now in position to prove that $H_\e: \cO_\d\to C^{0,\b-\a}(\R^{N-1})$ given by \eqref{eq:Geom-from-NMC-eps} is smooth.
\begin{proposition}
\label{prop:smooth-ovH}
For every $\e\geq 0$ and $\d>0$, the map  $H_\e: \cO_{\d} \subset C^{1,\beta}(\R^{N-1}) \to C^{0, \b-\a}(\R^{N-1})$ defined by (\ref{eq:Geom-from-NMC-eps}) is of class $C^\infty$, and for every  $k\in\N$, we have
 \begin{align}\label{eq:DkH-in-integ}
-\frac{\a}{2} D^k  H_\e({u})  (s) =&    \int_{\R^{N-1}}  D^kM_\e ({u},s,t)  d t- \int_{\R^{N-1}} D^k \ov  M_\e ({u},s,t)  d t ,
\end{align}
where $M_\e$ and $\ov M_\e$ are defined in \eqref{eq:def-M-oM}.  Moreover, there exists a positive constant $c=(\a,\b,\d,N )>1$ such  that  for every $\e\geq0$,
\begin{equation}
  \label{eq:est-Dk-cH}
\|D^k  H_\e({u})\|_{C^{0,\beta-\alpha}(\R^{N-1})} \leq  c(1+ \|{u}\|_{C^{1,\b}(\R^{N-1})} )^{c} .
\end{equation}
\end{proposition}

\begin{proof}
 With Lemma \ref{lem:est-cand-deriv} at hand, it suffices to follows precisely the arguments in \cite{CFW} to get the desired result, we skip the details.
\end{proof}
As a consequence of this result, we have
\begin{corollary}\label{cor:def-H-eps}
Let $k\in \N$, $u\in \cO$,  $u_1,\dots,u_k\in C^{1,\b}(\R^{N-1}) $. Then, for every   $s\in \R^{N-1}$, the map $\e \mapsto D^k H_\e(u)[u_1,\dots, u_k] (s)$ is continuous on $[0,1]$, and  $ D^k H_0(u)[u_1,\dots, u_k] (s)=  D^k H(u)[u_1,\dots, u_k] (s)$.
\end{corollary}
  \begin{proof}
The continuity of the map $\e\mapsto D^k H_\e(u)[u_1,\dots, u_k](\cdot) $ follows from Lemma \ref{lem:est-M-eps-ovM-eps} and the dominated convergence theorem. The last statement is an immediate consequence of \eqref{eq:DkH-in-integ}, Lemma \ref{lem:est-M-eps-ovM-eps} and the dominated convergence theorem.
%{eq:DkM}

  \end{proof}
\subsection{The derivative of the NMC operator}\label{ss:derivatives-of-H}
The aim of this section is to complete the proof of Proposition \ref{lem:diff-H} by computing the derivatives of the NMC operator $H$. To do so, we need to recall the expression of the NMC of a set $E$, in terms of principal integral,
$$
H_{E}(x)=PV\int_{\R^N}\frac{\t_{E_u^c}(y)}{ |x-y|^{ {N+\a} }}\,dy=\lim_{\e\to 0}\int_{|x-y|\geq \e}\frac{\t_{E_u^c}(y)}{|x-y|^{ {N+\a}}}\,dy,
$$
for every  $x \in  \de E$, where as before
$$
\tau_{E^c}(y):=1_{  E^c}(y) -1_{E}(y),
$$
with    $1_A$ denotes the characteristic function of $A$ and $E^c:=\R^N\setminus E$.  While we used the geometric expression to derive the regularity of the NMC operator \eqref{geometric-0}, we find  more convenient to use PV integral form to compute the full expressions of the linearzied operator about nonconstant functions.
 \begin{lemma}\label{lem:linearized-NMC}
For every $\l>0 $, $u\in \cO$ and
 $v,w\in C^{1,\b}(\R^{N-1})$, we have
\begin{align}\label{eq:DH-lamb-v}
  D H ( u)[v](s)&=2 PV\int_{\R^{N-1}}\frac{v(s)-v(s-t)}{(|t|^2+(u(s)-u(s-t))^2)^{\frac{N+\a}{2}}}\, dt\nonumber\\
 & \quad-2\int_{\R^{N-1}} \frac{v(s)+v(s-t)}{(|t|^2+(u(s)+u(s-t))^2)^{\frac{N+\a}{2}}}\, dt.
\end{align}
\end{lemma}
\begin{proof}
We consider $H_\e$ defined by \eqref{eq:Geom-from-NMC-eps}.  Then
recalling \eqref{eq:NMC-geom-eps},   we have
$$
 H_\e(u)(s) = - \frac{2}{\a}\int_{\partial E_u} \frac{(x-y)\cdot \nu_{E_u}(y)}{(|x-y|^2+\e^2)^{(N+\alpha)/2}} \,dy \quad \textrm{ for $x=(s,u(s)  )=F_u(s,1)$}.
$$
For $\varrho,\e>0$, we consider the map   $ {h}_{\e,\varrho}: \cO\to L^\infty(\R^{N-1})$   given by
$$
 {h}_{\e,\varrho}(u)(s) =\int_{\R^N}\frac{\t_{E_u^c}(y)}{( |x-y|^2+\e^2 )^{\frac{N+\varrho}{2}}}\,dy=\int_{\R^N}\frac{\t_{E_u^c}(y)}{( |(s,u(s)  )-y|^2+\e^2 )^{\frac{N+\varrho}{2}}}\,dy,
$$
with  $x=(s,u(s)  )=F_u(s,1)$.
For $\e>0$, we have that
$$
\nabla_y\cdot \frac{x-y}{(|x-y|^2+\e^2)^{(N+\alpha)/2}} = \frac{\alpha}{(|x-y|^2+\e^2)^{(N+\alpha)/2}} -\frac{(N+\alpha)\e^2}{(|x-y|^2+\e^2)^{(N+2+\alpha)/2}} .
$$
Multiplying this equality by $\t_{E_1^c}$,
integrating  on $\R^N$ and using the divergence theorem, we get
\be\label{eq:link-H-geomH-eps}
  H_\e(u)(s) = {h}_{\e,\a}(u)(s)- \frac{(N+\a)\e^2}{\a}   {h}_{\e,\a+2}(u)(s).
\ee
   By the change of variables $y=  u(\bar{s}) z$ and $\bar s=s-t$, we obtain
\begin{align*}
 {h}_{\e,\varrho}(u)(s)
&=  \int_{\R^{N-1}} \int_{\R}\frac{\t_{E_1^c}(\bar{s},z)}{(|s-\bar{s}|^2+ (\u(s) -u(\bar{s})z)^2 +\e^2)^{\frac{N+\varrho}{2}}}u(\bar{s})\,dz d\bar{s}\\
&=  \int_{\R^{N-1}} \int_{\R}\frac{\t_{(-1,1)^c}( z)}{(|t|^2+ (\u(s) -u(s-t)z)^2 +\e^2)^{\frac{N+\varrho}{2}}}u(s-t)\,dz d t.
\end{align*}
It is clear that  the map $ {h}_{\e,\varrho}: \cO\to C^{0,\b-\a}(\R^{N-1})$ is differentiable for every $\e,\varrho>0$.
 Now letting $V_\varsigma=u+\varsigma v$, we have
\begin{align}
D {h}_{\e,\varrho}(u)[v](s)=\frac{d {h}_{\e,\varrho} (V_\varsigma)}{d\varsigma}\Big|_{\varsigma=0}(s)&=  \int_{\R^{N-1}}v(\bar{s}) \int_{\R }\t_{(-1,1)^c}( z)  \Upsilon_\e(z) \, dz   d t \nonumber\\
&-  \int_{\R^{N-1}}    \int_{\R }\t_{(-1,1)^c}( z) \left\{ v(s)- v(s-t) z  \right\} \Upsilon_\e'(z)    dz d t,  \label{eq:Upsi-prim}
\end{align}
where
$$
\Upsilon_\e(z)=\frac{1}{(|t|^2+ (\u(s) -u(s-t) z)^2 +\e^2)^{\frac{N+\varrho}{2}}}.
$$
%which satisfies
%$$
%\Upsilon'(z)=( 3+\a)\frac{u(\bar{s}) (u(\bar{s}) z-\u(s))}{((s-\bar{s})^2+ (u(\bar{s}) z-\u(s))^2 +\e^2)^{\frac{5+\varrho}{2}}}.
%$$
Using integration by parts in \eqref{eq:Upsi-prim},    we deduce that
\begin{align*}
D {h}_{\e,\varrho}(u)[v](s)&= 2  \int_{\R^{N-1}}    \Bigl(   \left\{- v(s-t) -v(s)  \right\}  \Upsilon_\e(-1)   -  \left\{ v(s-t) -v(s)  \right\}  \Upsilon_\e(1) \Bigr) dt.
%
%&=   2  \int_{\R^{N-1}} \frac{v(s)-v(s-t) }{(|t|^2+ (\u(s) -u(s-t) )^2 +\e^2)^{\frac{N+\varrho}{2}}}   d t\
%M&\quad- 2  \int_{\R^{N-1}} \frac{v(s)+v(s-t) }{(|t|^2+ (\u(s) +u(s-t) )^2 +\e^2)^{\frac{N+\varrho}{2}}}   d t
\end{align*}
We then   conclude that
\begin{align}
D {h}_{\e,\varrho}(u)[v](s)&= 2  \int_{\R^{N-1}} \frac{v(s)-v(s-t) }{(|t|^2+ (\u(s) -u(s-t) )^2 +\e^2)^{\frac{N+\varrho}{2}}}   d t\nonumber\\
&\quad- 2  \int_{\R^{N-1}} \frac{v(s)+v(s-t) }{(|t|^2+ (\u(s) +u(s-t) )^2 +\e^2)^{\frac{N+\varrho}{2}}}  d t .
\label{eq:D-cH-eps-varrho}
\end{align}
We define
$$
\G_ v(s,t) = v(s)-v(s-t),\qquad \ov \G_ v(s,t) = v(s)+v(s-t) \qquad   \text{for $v \in  C^{1,\beta}(\R^{N-1})$, $\, s,t \in \R^{N-1}$}.
$$
Recalling \eqref{eq:defcKvarrho},  \eqref{eq:def-ov-cKvarrho} and \eqref{eq:def-Lamb-3}, we have
\begin{align}
D {h}_{\e,\varrho}(u)[v](s)&= 2  \int_{\R^{N-1}}|t|^{-N-\varrho} \G_v(s,t) \cK_{\varrho,\e}(u,s,t)  d t - 2  \int_{\R^{N-1}}  \ov \G_ v(s,t) \ov \cK_{\varrho,\e}(u,s,t) d t .
\label{eq:D-cH-eps-varrho}
\end{align}

We then deduce  from \eqref {eq:link-H-geomH-eps} and \eqref{eq:D-cH-eps-varrho} that, for every $\e>0$,
 \begin{align}
    & D {H}_{\e}( u)[v] (s)  =D{h}_{\e,\a}(u)[v](s)- \frac{(N+\a)\e^2}{\a}  D {h}_{\e,\a+2}(u)[v](s)\nonumber\\
    &
=     2  \int_{\R^{N-1}}|t|^{-N-\a} \G_v(s,t) \cK_{\a,\e}(u,s,t)  d t - 2  \int_{\R^{N-1}}  \ov \G_ v(s,t) \ov \cK_{\a,\e}(u,s,t) d t  \label{eq:Dh-eps-!}  \\
 &- \frac{(N+\a)\e^2}{\a}\left( 2  \int_{\R^{N-1}}|t|^{-N+2-\a} \G_v(s,t) \cK_{\a+2,\e}(u,s,t)  d t -  2 \int_{\R^{N-1}}  \ov \G_ v(s,t) \ov \cK_{\a+2,\e}(u,s,t) d t\right) .\nonumber
\end{align}
For $\e\in(0,1]$ and $s$ fixed,  we let
 \be\label{eq:def-B_s-eps-h}
B_{s,\a}(\e):= 2\int_{\R^{N-1}}|t|^{-N-\a} \G_v(s,t) \cK_{\a,\e}(u,s,t)  d t -  2 \int_{\R^{N-1}}  \ov \G_ v(s,t) \ov \cK_{\a,\e}(u,s,t) d t .
\ee
We claim that $B_{s,\a}$ is  smooth on $(0,1]$ and extends continuously at  $\e=0$, for every fixed $s$. Indeed, it is clear that $\e\mapsto  \int_{\R^{N-1}}  \ov \G_ v(s,t)  \ov \cK_{\a,\e}(u,s,t) d t$  is smooth on $[0,1]$ by the dominated convergence theorem, since $u\in\cO$. Next, we write
\begin{align}
2 \int_{\R^{N-1}}|t|^{-N-\a} \G_v(s,t) \cK_{\a,\e}(u,s,t)  d t&= \int_{\R^{N-1}}|t|^{-N-\a} (\G_v(s,t) +\G_v(s,-t) )\cK_{\a,\e}(u,s,t)  d t \nonumber\\
 & + \int_{\R^{N-1}}|t|^{-N-\a}( \G_v(s,t)-\G_v(s,-t) )\cK_{\a,\e}(u,s,t)  d t. \label{eq:Gam-even-od}
\end{align}
We then  observe that
\begin{align}
|\G_v(s,t) +\G_v(s,-t) |& =|v(s+t)-2v(s)+v(s-t)|\nonumber\\
&=|\int_{0}^1\left(\n v(s+\rho t)\cdot t- \n v(s-\rho t)\cdot t \right) d\rho| \nonumber\\
&\leq 2 \|v\|_{C^{1,\b}(\R^{N-1}) }\min (|t|,|t|^{1+\b}).
\end{align}
Recalling that $\b>\a$, and since $\cK_{\a,\e}(u,s,t)\leq 1$, it then  follows from the dominated convergence theorem  that, for every fixed $s\in \R^{N-1}$, the function $$\e\mapsto  \int_{\R^{N-1}}|t|^{-N-\a} (\G_v(s,t) +\G_v(s,-t) )\cK_{\a,\e}(u,s,t)  d t $$ is  smooth on $(0,1]$    and has a finite limit at $\e\to0$. Now, we have
\begin{align*}
 \int_{\R^{N-1}}&|t|^{-N-\a}( \G_v(s,t)-\G_v(s,-t) )\cK_{\a,\e}(u,s,t)  d t\\
 &\qquad\quad=\frac{1}{2} \int_{\R^{N-1}}|t|^{-N-\a}( \G_v(s,t)-\G_v(s,-t) )(\cK_{\a,\e}(u,s,t)-\cK_{\a,\e}(u,s,-t) )  d t.
\end{align*}
Since  $| \G_v(s,t)-\G_v(s,-t) |\leq \min(|t|,1) $, by \eqref{eq:est-K-odd}, we have
$$
|(   \G_v(s,t)-\G_v(s,-t)  )(\cK_{\a,\e}(u,s,t)-\cK_{\a,\e}(u,s,-t))|\leq c \min (|t|^{\b+1},|t|).
$$
Therefore the function    $\e \mapsto  \int_{\R^{N-1}}|t|^{-N-\a}( \G_v(s,t)-\G_v(s,-t) )\cK_{\a,\e}(u,s,t)  d t$ is smooth on $(0,1]$    and also has a finite limit as $\e\to0$, by the dominated convergence theorem.  The claim is thus proved.

For $s\in \R^{N-1}$ fixed and $\e\in (0,1]$, we now have,  see \eqref{eq:Dh-eps-!} and \eqref{eq:def-B_s-eps-h},
 \begin{align*}
    D {H_\e}( u)[v ](s)
% & = B_{s,\a}(\e)   - \frac{N+\a}{\a} B_{s,\a+2}(\e)   \\
 %
 &=B_{s,\a}(\e)  + \frac{\e}{\a}  \frac{d}{d\e}  B_{s,\a}(\e)  \\
 %
% &=\frac{\a-1}{\a}(h_{\e,\a}(s) - \ov{h}_{\e,\a}(s  ) )+ \frac{1}{\a}  \frac{d}{d\e}\left( \e(h_{\e,\a}(s) - \ov{h}_{\e,\a}(s  ))\right )\\
 %
 &= \frac{\a-1}{\a} B_{s,\a}(\e) + \frac{1}{\a}  \frac{d}{d\e}\left( \e B_{s,\a}(\e)\right ) .
\end{align*}
Integrating this from $0$ to $\ov{\e}\in (0,1]$,
 we  find that
 \begin{align}\label{eq:int-Dh-eps}
 \int_0^{\ov{\e}}   D {H_\e}( u)[v ](s) d\e & =  \frac{\a-1}{\a}\int_0^{\ov{\e}} B_{s,\a}(\e)\, d\e+ \frac{1}{\a} \left(\ov{\e}B_{s,\a}(\ov{\e})-\lim_{\e\to 0} \e B_{s,\a}(\e)\right).
\end{align}
From the continuity of $B_{s,\a}$  on $[0,1]$,  we deduce that
$$
   \int_0^{\ov{\e}}   D {H_\e}( u)[v](s) d\e  =  \frac{\a-1}{\a}\int_0^{\ov{\e}} B_{s,\a}(\e) d\e+ \frac{1}{\a} \ov{\e}B_{s,\a}(\ov{\e}).
$$
Dividing this equality by $\ov{\e}$ and letting $\ov{\e}\to 0$, we then have
$$
   \lim_{\ov{\e}\to 0}   D {H_{\ov{\e}}}( u)[v]  (s)   =  \frac{\a-1}{\a}\lim_{\ov{\e}\to 0}  B_{s,\a}(\ov{\e}) + \frac{1}{\a} \lim_{\ov{\e}\to 0}  B_{s,\a}(\ov{\e})= \lim_{\ov{\e}\to 0} B_{s,\a}(\ov{\e}),
$$
where we have used the  continuity of $\ov \e\mapsto  D H_{\ov{\e}}( u)[v]  (s)$  on $[0,1]$, by Corollary \ref{cor:def-H-eps}. Since $ D H_{\ov{\e}}( u)[v](s) \to  D H ( u)[v](s) $ as $\ov{\e}\to 0 $  by  Corollary \ref{cor:def-H-eps}, we thus have
\begin{align*}
&  D {H }( u)[v] (s)   = \lim_{\ov{\e}\to 0}  B_{s,\a}(\ov{\e})\\
&  =  2\lim_{\ov{\e}\to 0}\int_{\R^{N-1}}|t|^{-N-\a} \G_v(s,t) \cK_{\a,\ov{\e}}(u,s,t)  d t -  2 \int_{\R^{N-1}}  \ov \G_ v(s,t) \ov  \cK_{\a,0}(u,s,t) d t\\
  &=2 PV\int_{\R^{N-1}}\frac{v(s)-v(s-t)}{(|t|^2+(u(s)-u(s-t))^2)^{\frac{N+\a}{2}}}\, dt-2\int_{\R^{N-1}} \frac{v(s)+v(s-t)}{(|t|^2+(u(s)+u(s-t))^2)^{\frac{N+\a}{2}}}\, dt.
\end{align*}
This completes the proof.
\end{proof}

\end{document}